\documentclass[11pt, reqno]{amsart}
\usepackage{amssymb,amscd,amsmath}
\usepackage{verbatim} 
\usepackage{todonotes}
\usepackage{enumerate}
\usepackage{fullpage}

\usepackage{amsthm}

\newtheorem{thm}{Theorem}[section]
\newtheorem{lem}[thm]{Lemma}
\newtheorem{prop}[thm]{Proposition}
\newtheorem{cor}[thm]{Corollary}
\newtheorem{rmk}[thm]{Remark}
\newtheorem{dfn}[thm]{Definition}

\theoremstyle{remark}

\renewcommand{\Re}{\text{Re}}

	\newcommand{\R}{\mathbb{R}}			
	\newcommand{\N}{\mathbb{N}} 			
	\newcommand{\C}{\mathbb{C}}			
	\newcommand{\Z}{\mathbb{Z}}				
	\newcommand{\Q}{\mathbb{Q}}       

	\newcommand{\abs}[1]{\left|#1\right|} 		

	\newcommand{\sabcd}[4]{\left(\begin{smallmatrix}#1&#2\\#3& #4\end{smallmatrix}\right)}

\DeclareMathOperator{\SL}{SL}

	\title{Products of Eisenstein series and Fourier expansions of modular forms at cusps}
	\author{Martin Dickson and Michael Neururer}
	\usepackage{hyperref}
	\hypersetup{pdfauthor={Martin Dickson and Michael O. Neururer},%
	            pdftitle={Products},%
	            pdfsubject={Number Theory},%
	            pdfcreator={XeLaTeX},%
	            colorlinks=true,%
	            urlcolor=blue,
	            linkcolor=[rgb]{0,.6,1},
	            citecolor=blue}

\begin{document}

\begin{abstract} We show, for levels of the form $N = p^a q^b N'$ with $N'$ squarefree, that in weights $k \geq 4$ every cusp form $f \in \mathcal{S}_k(N)$ is a linear combination of products of certain Eisenstein series of lower weight.  In weight $k=2$ we show that the forms $f$ which can be obtained in this way are precisely those in the subspace generated by eigenforms $g$ with $L(g, 1) \neq 0$. As an application of such representations of modular forms we can calculate Fourier expansions of modular forms at arbitrary cusps and we give several examples of such expansions in the last section.\end{abstract}

\maketitle

\section{Introduction}\label{sctn:introduction}

The space $\mathcal{M}_k(N,\chi)$ of modular forms of level $\Gamma_0(N)$, weight $k$ and nebentypus $\chi$ splits into the direct sum of the Eisenstein subspace $\mathcal{E}_k(N,\chi)$ and the space of cusp forms $\mathcal{S}_k(N,\chi)$.  It is straightforward to compute Fourier expansions and Hecke eigenforms in the Eisenstein subspace, but the space of cusp forms is far more mysterious, and any method of generating cusp forms is therefore of great interest.  In this article we examine one of the simplest methods of generating cusp forms: What is the subspace of $\mathcal{S}_k(N,\chi)$ generated by (the cuspidal projection of) products of Eisenstein series of lower weight?\\

For $N=1$ the answer to this question is very well-known: the graded ring\begin{footnote}{When $\chi = \mathbf{1}_N$ is the principal character modulo $N$ we write $\mathcal{M}_k(N)$ for $\mathcal{M}_k(N, \mathbf{1}_N)$.}\end{footnote} $\oplus_{k \geq 0} \mathcal{M}_k(1)$ is a polynomial ring with two generators, one in degree four and one in degree six, corresponding to the Eisenstein series $E_4$ and $E_6$. This means that every cusp form of level $N=1$ is a linear combination of products of Eisenstein series.  However the number of products required to form the monomials in $E_4$ and $E_6$ for these linear combinations grows linearly with the weight $k$, which means these monomials are rather complicated.  It is therefore natural to ask whether one can have simpler products at the expense of taking more Eisenstein series.  Pushing this to the extreme we are led to ask: What is the subspace of $\mathcal{M}_k(N, \chi)$ generated by Eisenstein series and products of \textit{two} Eisenstein series?\\ 

Using the Rankin--Selberg method (as observed in \cite{Zagier1977} \S5) one can show that, for $k \geq 8$, $\mathcal{M}_k(1)$ can be generated by the products $E_l E_{k-l}$ for $4 \leq l \leq k-4$. Similar statements are known to hold for $\mathcal{M}_k(p)$ for $p$ prime and $k \geq 4$ (see Imago\={g}lu--Kohnen \cite{ImamogluKohnen2005} for $p=2$, Kohnen--Martin \cite{KohnenMartin2008} for $p>2$). The most complete result in this direction was found by Borisov--Gunnells in \cite{BorisovGunnells2001toricvarieties}, \cite{BorisovGunnells2001}, and \cite{BorisovGunnells2003}. They show that for weights greater than two and any level $N\geq 1$ the whole space $\mathcal{M}_k(\Gamma_1(N))$ is generated by $\mathcal{E}_k(\Gamma_1(N))$ and products of two \textit{toric} Eisenstein series $\tilde{s}_{a/N}^{(k)}$ for $a\in\{0,\ldots,N-1\}$, while for $k=2$ one only obtains a subspace, $\mathcal{S}_{2,\mathrm{rk}=0}(N)+\mathcal{E}_k(N,\chi)$, where $\mathcal{S}_{2,\mathrm{rk}=0}(N)$ is defined below. 

The main application we want to present is the calculation of Fourier expansions at arbitrary cusps. While the toric Eisenstein series of Borisov--Gunnells have remarkably simple rational Fourier expansions at $\infty$, the Fourier expansions at other cusps are harder to obtain. This led us to consider instead the well-studied Eisenstein series
\begin{equation}\label{eqn:eis-series-def}
E_{l}^{\phi,\psi}(z)=e_l^{\phi,\psi} + 2\sum_{n\geq 1}
\sigma_{l-1,\phi,\psi}(n)q^n\in\mathcal{M}_l(M,\phi\psi),
\end{equation}
where $q = e^{2\pi i z}$, $\phi$ and $\psi$ are primitive characters of level $M_1$ and $M_2$, $M_1M_2=M\mid N$, $\sigma_{l-1, \phi, \psi}(n) = \sum_{d \mid n}\phi(n/d)\psi(d)d^{l-1}$, and the constant term $e_{l}^{\phi, \psi}$ (either zero or a value of a Dirichlet $L$-function) is given in Section \ref{sctn:generating-cusp-forms}. The advantage of working directly with the Eisenstein series in \eqref{eqn:eis-series-def} is that their Fourier expansions at cusps other than $\infty$ are comparatively easy to obtain and were explicitly calculated by Weisinger \cite{Weisinger1977} (we use a corrected version by Cohen \cite{Cohen2017unpublished}).

Before we describe our results we mention a different, rather general recent result by Raum \cite{Raum2016}: Let $k \geq 8$ be an integer, let $\rho$ be a representation of $\SL_2(\Z)$ on a complex vector space $V$ such that $\ker(\rho)$ contains a congruence subgroup, and define $\mathcal{M}_k(\rho)$ to be the space of $V$-valued functions transforming as modular forms for the automorphy factor $\gamma \mapsto (cz + d)^{-k} \rho(\gamma^{-1})$.  Then
\begin{equation}\label{eqn:raums-generators}
\mathcal{M}_k(\rho) = \mathcal{E}_k(\rho)+\text{span}_{\phi:\rho_M\otimes\rho_{M'}\rightarrow \rho}
\left(T_M E_l\otimes T_{M'}E_{k-l}\right),
\end{equation}
where $4 \leq l \leq k-4$, $\rho_M$ is the permutation representation on $\Gamma_0(M) \backslash \SL_2(\Z)$, the $E_k$ are corresponding vector-valued Eisenstein series, and the $T_M$ are certain natural vector-valued Hecke operators.

In order to state our main theorem, when the character $\chi$ is trivial, let us first define the space of products more precisely. Let $B_d$ be the lifting operator that associates to a modular form $f$ of weight $k$ the form $f|B_d(\tau) = d^{\frac{k}{2}}f(d\tau)$. Let $k \geq 2$ be even and fix a positive integer $N$. We then define $\mathcal{Q}_k(N) \subset \mathcal{M}_k(N)$ to be the subspace generated by all products of the form
\[
E_{l}^{\phi, \psi}|B_{d_1d}\cdot E_{k-l}^{\overline{\phi}, \overline{\psi}} | B_{d_2d}
\]
that lie in $\mathcal{M}_k(N)$ with the additional condition that $d_1 M_1$ divides the squarefree part of $N$. In other words $l\in\{1,\ldots,k-1\}$, $\phi$ and $\psi$ are primitive characters modulo $M_1,M_2$ with $\phi \psi(-1) = (-1)^l$, and $d_1,d_2$ are integers such that $d_1 M_1 d_2 M_2d \mid N$. We only need to require $(\phi, \psi, l) \neq (\mathbf{1}, \mathbf{1}, 2), (\mathbf{1}, \mathbf{1}, k-2)$, where $\textbf{1}$ is the trivial character, since these choices do not produce modular forms. The additional condition on $d_1 M_1$ implies that $(d_1 M_1, d_2 M_2 d) = 1$. Our main result is:

\begin{thm}\label{intro:prods-full-space}
Let $k \geq 4$ be even.  Let $N=p^a q^b N'$ where $p$ and $q$ are primes, $a, b \in \Z_{\geq 0}$, and $N'$ is squarefree. Then the restriction of the cuspidal projection to $\mathcal{Q}_k(N)$ is surjective, i.e.
\[
\mathcal{M}_k(N) =  \mathcal{Q}_k(N)+\mathcal{E}_k(N).
\]
\end{thm}

In particular, if $N=p^a q^b$ we only need $\mathcal{E}_k(N)$ and products of the form $E_{l}^{\textbf{1}, \psi}|B_{d}\cdot E_{k-l}^{\textbf{1}, \overline{\psi}} | B_{d_2d}$ to generate $\mathcal{M}_k(N)$. This is a minor improvement to  \cite{BorisovGunnells2001toricvarieties},\cite{BorisovGunnells2001}, and \cite{BorisovGunnells2003}, from which one can deduce that $\mathcal{E}_k(N)$ and products of the form
$E_{l}^{\textbf{1}, \psi}|B_{d_1}\cdot E_{k-l}^{\textbf{1},\overline{\psi}} | B_{d_2}$ generate $\mathcal{M}_k(N)$. An advantage of our proof is that it does not use the theory of toric modular forms from which Borisov--Gunnells draw a powerful Hecke stability result. Instead we use a vanishing criterion for cusp forms with many vanishing $L$-values, Theorem \ref{thm:eichler-shimura-paqb}, which could be of independent interest.

The case of weight $2$ is different: Indeed, one sees immediately from the Rankin--Selberg method that the products of Eisenstein series are orthogonal to every newform $f$ with vanishing central $L$-value, i.e. $L(f,1)=0$.  Accordingly we define the space $\mathcal{S}_{2,\mathrm{rk}=0}(N)$ to be generated by newforms and lifts of newforms with non-zero central $L$-value.  We obtain the analogue of Theorem \ref{intro:prods-full-space} subject to this constraint.

\begin{thm}\label{intro:prods-full-space-k=2}  Let $N$ and $\mathcal{Q}_2(N)$ be as in Theorem \ref{intro:prods-full-space}.  Then
\[
\mathcal{S}_{2,\mathrm{rk}=0}(N) \oplus \mathcal{E}_2(N)= \mathcal{Q}_2(N)+\mathcal{E}_2(N).
\]
\end{thm}

This phenomenon of isolating $\mathcal{S}_{2,\mathrm{rk}=0}(N)$ is also observed by Borisov--Gunnells \cite{BorisovGunnells2001}. \\

We develop much of the theory to allow for more general level than $N=p^a q^b N'$ and will discuss this restriction for the level $N$ below.  We also point out how similar results can be obtained when the character $\chi$ is non-trivial by proving the analogue of Theorem \ref{intro:prods-full-space-k=2} for $\mathcal{S}_2(p, \chi)$, see Theorem \ref{thm:prods-gen-weight-2-prime}.  Before we give an idea for the proof of the theorems we give a few explicit examples, and highlight some of the applications of such an expression for a newform:

\begin{enumerate}
\item $N=1,k=12$: The most well-known example is of course the discriminant modular form, which in our normalisation becomes
\[
\Delta = \frac{50}{3}E_4^{\textbf{1},\textbf{1}}E_8^{\textbf{1},\textbf{1}}-\frac{147}{4}(E_6^{\textbf{1},\textbf{1}})^2.
\]
\item\label{intro:example-2} $N=11, k=2$: Let $\phi$ be the character modulo $11$ that maps $2$ to $\zeta_{10}$ and $f_{11} \in \mathcal{S}_2(11)$ be the unique newform in this space. Then
\[
f_{11} = \left(\frac{1}{\sqrt{5}} - \frac{1}{4}\right)E_1^{\textbf{1},\phi}E_1^{\textbf{1},\overline{\phi}}-\left(\frac{1}{\sqrt{5}} + \frac{1}{4}\right)E_1^{\textbf{1},\phi^3}E_1^{\textbf{1},\overline{\phi}^3}.
\]
\item\label{intro:example-3} $N=32, k=2$: Let $\chi_4$ be the primitive character modulo $4$ and let $f_{32}= q - 2q^5 -3q^9 + 6q^{13} +2q^{17}+ O(q^{20}) \in \mathcal{S}_2(32)$ be the unique newform in this space. Then
\begin{equation*}
f_{32} = -\frac12 E_1^{\textbf{1},\chi_4}\cdot E_1^{\textbf{1},\chi_4}|B_4 + \frac{1}{\sqrt{2}}E_1^{\textbf{1},\chi_4}\cdot E_1^{\textbf{1},\chi_4}|B_8 +\frac{1}{2\sqrt{2}}E_1^{\textbf{1},\chi_4}|B_2\cdot E_1^{\textbf{1},\chi_4}|B_4 -\frac{1}{2}E_1^{\textbf{1},\chi_4} |B_2\cdot E_1^{\textbf{1},\chi_4}|B_8.
\end{equation*}
\end{enumerate}

An expression of a modular form $f$ as a sum of products of Eisenstein series provides a way of calculating the Fourier expansion of $f$ at $\infty$.  Once such an expression for $f$ is obtained, $O(n \log(n))$ operations are required to compute $n$ Fourier coefficients of $f$ which is theoretically best possible,\begin{footnote}{We thank A. Booker for pointing this out.}\end{footnote} and also appears to work well in practice.\begin{footnote}{See \url{http://mathoverflow.net/q/221781/} for an example computed by D. Loeffler}\end{footnote}  \\

Moreover, as mentioned in \cite{Raum2016}, one can use such expressions to compute Fourier expansions at \textit{any} cusp of $\Gamma_0(N)$. For square-free $N$ one can deduce the Fourier expansion of a modular form $f\in\mathcal{M}_k(N)$ at any cusp from the one at $\infty$ by using Atkin-Lehner operators. However for general $N$ the Fourier expansions at certain cusps are difficult to access yet carry important information. We have implemented an algorithm in Sage \cite{sage} that calculates Fourier expansions at any cusp of $\Gamma_1(N)$ based on a representation of the modular form as a linear combination of products of Eisenstein series. For example we can use the expansion in example \eqref{intro:example-3} to obtain the Fourier expansion of $f_{32}$ at the cusp $\frac{1}{8}$. For that purpose we first choose $\gamma = \sabcd{1}{0}{8}{1}\in\SL_2(\Z)$ that sends $\infty$ to $1/8$. The expansion is then given by 
\[
f_{32}|\gamma = -iq + 2iq^5 + 3iq^9 - 6iq^{13} - 2iq^{17} + O(q^{20})=-if_{32}.
\]

Using the Fourier expansions of a newform $f$ at cusps we can also compute the root number of $f$ and eigenvalues of Atkin-Lehner operators.  This furnishes another example of an important datum which cannot immediately be read from the Fourier expansion of $f$ at $\infty$ when the level is not squarefree. In Section \ref{scn:Fourier expansions} we give several more examples of Fourier expansions of modular forms of levels that involve high prime powers, where we also calculate Atkin-Lehner eigenvalues. The algorithms we use are available at \cite{Github_mneururer}.
\\

In future work \cite{DicksonNeururer17} we plan to use this method of computing Fourier expansions at cusps other than $\infty$ to study the local representation-theoretic aspects of newforms.  The connection between Fourier expansions at cusps and local components of newforms is explained in \cite{Brunault2016} and \cite{CorbettSaha2017}.  In brief, the Fourier expansions at cusps of a newform $f$ can be used to obtain values of the new vector in the local Whittaker model.  From this one can extract root numbers of twists of the local representation, and also determine the local component of $f$.\\

Let us now give a sketch of proof of Theorem \ref{intro:prods-full-space} (the proof of Theorem \ref{intro:prods-full-space-k=2} requires minor modifications). Note that in Section \ref{sctn:generating-cusp-forms} we argue for the most part with a space $P_k(N)$ instead of $\mathcal{Q}_k(N)$, and in \S \ref{sctn:new-part} we show that $P_k(N)$ has the same projection to $\mathcal{S}_k^{\text{new}}(N)$ as $\mathcal{Q}_k(N)$. Theorem \ref{intro:prods-full-space} then follows by induction from the fact that the projection of $P_k(N)$ to $\mathcal{S}_k^{\text{new}}(N)$ equals $\mathcal{S}_k^{\text{new}}(N)$, which is the statement of Theorem \ref{thm:prods}. For the proof of Theorem \ref{thm:prods} suppose to the contrary, that the projection of $P_k(N)$ cuts out a proper subspace of $\mathcal{S}_k^{\text{new}}(N)$. We may pick a non-zero $G \in \mathcal{S}_k^{\text{new}}(N)$, orthogonal to $P_k(N)$.  A standard application of the Rankin--Selberg method (\S \ref{sctn:generating-cusp-forms}) allows one to see that, if $G$ were a \textit{newform}, then all the critical $L$-values $L(G_\psi | W_S^{NM}, j)$ would have to vanish for all primitive characters $\psi$ of conductor dividing $N$ and all sets $S$ of prime divisors of $N'$ (except for some cases when $j=2, k-2$, when the technical difficulties coming from weight two Eisenstein series enter). Here $G_\psi$ is the twist of $G$ by $\psi$. However, since our $P_k(N)$  might not be closed under the action of the Hecke operators, we cannot assume that $G$ is a newform.  On the other hand, our space $P_k(N)$ will, by construction, be closed under the action of the Atkin--Lehner operators $W_\ell^N$ for $\ell|N'$ and so we can at least assume that $G$ is an eigenfunction of all these operators.  Moreover we can modify $G$ so that the orthogonality of $G$ to $P_k(N)$ is equivalent to the vanishing of many twisted $L$-values of $G$. We prove a general statement, possibly of independent interest, that if $G$ is a cusp form which is an eigenfunction of certain Atkin--Lehner operators and for which sufficiently many twisted $L$-values vanish, then $G=0$.  The argument proceeds via modular symbols, and extends a result of Merel who proves a similar vanishing criterion in the case when $G$ is a newform.\\

The reason the assumption $N = p^a q^b N'$ enters is because we want to be in a situation where,  if $f$ is a newform (or a sum of newforms with the same $W_\ell^N$-eigenvalue for all $\ell \in T$) and $\alpha$ is a primitive character modulo $M \mid N$, then the $W_S^{NM}$ (pseudo-)eigenvalues of the twists $f_\alpha$ for each set $S$ of prime divisors of $N/M$ are determined by those of $f$. With our methods, this conditon arises naturally in the proof of Theorem \ref{intro:prods-full-space}, and our argument would extend immediately to any situation where it holds. 
When $N$ is squarefree this condition is automatic, since the twisting and Atkin--Lehner operators must commute (c.f. Proposition \ref{prop:twists-commute-with-al-and-hecke}).  When $N$ is not squarefree this is a much more difficult question, and it seems unlikely that a purely local argument will work.  Indeed our extension to level $N = p^a q^b N'$ stems from a rather different argument using the modular symbols relations, which allows us to to avoid this condition altogether when the number of prime factors of $N$ is restricted as above.\\

Let us finish by remarking that one can easily compute (see \S\ref{sctn:generating-cusp-forms}) a trivial upper bound $p_k(N)$ for the number of generators of $Q_k(N)$.  We compare this to the dimension $\dim \mathcal{S}^{\text{new}}_k(N)$ in the ``squarefree'' and ``prime-power'' level aspects.  In both cases the result is that $P$ grows quicker than $\dim \mathcal{S}_k(N)$, although not by much, particularly in the prime-power case.  When $k=2$ the problem of improving the upper bound $p_k(N)$ is interesting because of potential applications to the conjecture of Brumer on the number of newforms of level $N$ for which $L(f, 1) \neq 0$.\\

\section{Preliminaries}\label{sctn:preliminaries}

Let $N, k \in \Z$ be a positive integers, and let $\chi$ be a Dirichlet character modulo $N$.  We keep the notations from the introduction for spaces of modular forms; we tacitly assume that $\chi(-1) = (-1)^k$ since otherwise these spaces are zero.  Our normalisation of the slash operator is
\[\left(f |_k \begin{pmatrix} a & b \\ c & d \end{pmatrix}\right)(z) = \frac{(ad-bc)^{k/2}}{(cz+d)^k} f\left(\frac{az + b}{cz+d}\right),\]
so that diagonal matrices act trivially.  We write $|$ for  $|_k$ since the weight $k$ will be clear from the context.\\

We denote by $\mathbf{1}_N$ the principal character modulo $N$, which satisfies $\mathbf{1}_N(n) = 1$ for $(n, N) = 1$, and $\mathbf{1}_N(n) = 0$ otherwise.  We write $\mathbf{1}$ for the trivial character, which satisfies $\mathbf{1}(n) = 1$ for all $n$.  Any character $\chi$ modulo $N$  can be factorised as a product $\chi = \prod_{p|N} \chi_p$ over the prime divisors of $p|N$, where $\chi_p$ is a uniquely determined character modulo $p^{v_p(N)}$. If $S$ is a set of prime divisors of $N$, then we write $\chi_S=\prod_{p\in S}\chi_p$ for the $S$-part of $\chi$.  For a set $S$ of prime divisors of $N$ and a divisor $M$ of $N$, we write $M_S$ for the $S$-part of $M$, i.e. $\prod_{p\in S}p^{v_p(M)}$. We will also use the notation $S_M=\{p\in S; p\mid M\}$ and $\overline{S} = \{p \mid N; p\text{ prime}\} \setminus S$ (we will clarify the dependence on $N$ when confusion may arise). With this notation we have $M_SM_{\overline{S}}=M$ for any divisor $M|N$, in particular $N_SN_{\overline{S}}=N$.\\

For primes $p$ with $(p, N)=1$, we write $T_p$ for the Hecke operators on $\mathcal{M}_k(N, \chi)$; these are extended multiplicatively to $T_n$ for $(n, N)=1$.  When $q \mid N$ we write $U_q$ for the Hecke operators extended from the operators $U_p$ (where $p$ is a prime divisor of $N$); the normalisation is
\[f | U_p = p^{k/2-1} \sum_{j=0}^{p-1} f | \begin{pmatrix} 1 & j \\ 0 & p\end{pmatrix}.\]

For a set of prime divisors $S$ of $N$ we define the Atkin--Lehner involution
\[
W_S^N=\begin{pmatrix}
N_Sx & y \\ Nz & N_Sw
\end{pmatrix}\in\text{M}_2(\Z),
\]
where $y\equiv 1\pmod {N_S}$, $x\equiv 1\pmod {N_{\overline{S}}}$ and $\det W_S^N=N_S$.  If $M$ is a divisor of $N$, then we sometimes use the notation $W_M^N$ for $W_S^N$ with $S = \{p \mid M\}$.  We simply write $W_N$ for $W_N^N=\left(\begin{smallmatrix} 0 & 1 \\ -N & 0\end{smallmatrix}\right)$.  The following properties of $W_S^N$ are well-known (see for example \cite[\S 1]{AtkinLi78}):  
\begin{prop}\label{prop:al}  \begin{enumerate}[(i)]
\item Let $S$ be a set of prime divisors of $N$.  If
\[\begin{aligned}
M= \begin{pmatrix}
N_Sx' & y' \\ Nz' & N_Sw'
\end{pmatrix}
\end{aligned}\]
is any matrix with $x',y',z',w'\in\Z$ of determinant $N_S$ then
\begin{align}\label{eq:other-atkin-op}
f| M = \overline{\chi}_S(y')\overline{\chi}_{\overline{S}}(x')f| W^N_S.
\end{align}
In particular, $W_S^N$ does not depend on the choice of $x, y, z, w$.
\item Let $f\in \mathcal{M}_k(N,\chi)$. Then
\[
f| W_S^N \in \mathcal{M}_k(N,\overline{\chi}_S\chi_{\overline{S}}),
\]
and $W_S^N$ preserves the subspace of cusp forms. 
\item If $S$ and $S'$ are disjoint sets of prime divisors of $N$, then
\[
f|W_{S\cup S'}^N = \chi_{S'}(N_S)(f|W_S^N)|W_{S'}^{N}.
\]
We also have
\begin{align}\label{eq:al-twice}
f | W_S^N | W_S^N = \chi_S(-1)\overline{\chi}_{\overline{S}}(N_S)f.
\end{align}
\item The adjoint of $W_S^N$ on $\mathcal{M}_k(N,\chi)$ with respect to the Petersson inner product is given by
\[
W_S^{N,*}=\chi_S(-1)\chi_{\overline{S}}(N_S)W_S^N.
\]
\item Let $p$ be a prime divisor of $N$ with $p \notin S$. Then
\[
f|U_p|W_S^N = \chi_S(p)f|W_S^N|U_p.
\]
\end{enumerate}\end{prop}

By a newform, we mean an element $f \in \mathcal{S}_k(N, \chi)$ which is an eigenfunction of all Hecke operators, normalised to have first Fourier coefficient equal to one.  We write $\mathcal{S}_k^{\text{new}}(N, \chi)$ for the subspace of $\mathcal{S}_k(N, \chi)$ generated by the newforms, so $f \in \mathcal{S}_k^{\text{new}}(N, \chi)$ is a linear combination of newforms; we refer to these as elements of the new subspace. If $\chi$ can be defined modulo $N/q$, then any newform in $\mathcal{S}_k(N,\chi)$ is an eigenfunction of the operator $W_q^N$, see Proposition 13.3.11 in \cite{CohenStromberg2017}.\\

Let $q$ be a prime divisor of $N$. On the new subspace there is a close connection between the Hecke operator $U_q$ and the Atkin-Lehner operator $W_q^N$:
\begin{prop}\label{prop:al-hecke-connection}[\cite{CohenStromberg2017} Proposition 13.3.14]
Let $\chi$ be a Dirichlet character modulo $N$ and suppose it is induced by a character modulo $N/q$ (or equivalently that $\chi_q$ is principal). Let $f$ be a newform of $\mathcal{S}_k(N,\chi)$ with $q$-th Fourier coefficient $a_q$ and Atkin-Lehner eigenvalue $\lambda_q(f)$.
\begin{itemize}
\item If $q^2\mid N$ then $a_q=0$.
\item If $q \mid N$ but $q^2\nmid N$ then $\lambda_q(f) = -q^{1-\frac{k}{2}}a_q$ and hence we have the equality of operators
\[
W_q^N = -q^{-\frac{k}{2}+1}U_q.
\]
on $\mathcal{S}_k^{\text{new}}(N, \chi)$.
\end{itemize}
\end{prop}

The third class of operators that play a major role for us are various twisting operators.  Let $f \in \mathcal{S}_k(N, \chi)$ with Fourier expansion $f(z) = \sum_{n \geq 1} a_n e(n z)$, let $\alpha$ be a Dirichlet character modulo $M$, and define
\[f_\alpha(z) = \sum_{n \geq 1} a_n \alpha(n) e(nz).\]
With $\alpha, f$ as above, define also
\[S_\alpha (f) = \sum_{a \bmod M} \overline{\alpha(a)} f|_k \begin{pmatrix} 1 & a/M \\ 0 & 1\end{pmatrix}.\]
Note that if $\alpha$ is primitive modulo $M$ we have
\begin{equation}\label{eqn:twists-for-prim-char}S_\alpha (f) = G(\overline{\alpha}) f_\alpha,\end{equation}
where $G(\overline{\alpha})$ is the Gauss sum of $\alpha$. For any $z \in \mathfrak{H}$ we can view the function $n' \mapsto \left(f|_k \left(\begin{smallmatrix} 1 & n'/N' \\ 0 & 1 \end{smallmatrix}\right)\right)(z)$ as a function $F : (\Z/N'\Z)^\times \to \C^\times$, and we see by Fourier inversion that
\begin{equation}\label{eqn:fourier-inversion} f|_k \begin{pmatrix} 1 & n'/N' \\ 0 & 1 \end{pmatrix} = \sum_{\alpha \bmod N'}\frac{\alpha(n')}{\varphi(N')} S_\alpha(f),\end{equation}
the sum being over all Dirichlet characters modulo $N'$.\\

Finally we state some standard facts about the commutation relations for the above operators in the cases we will need them.  These can be proved by direct computation (see also \cite{AtkinLi78} \S 3).

\begin{prop}\label{prop:twists-commute-with-al-and-hecke}  Let $N \in \N$, let $f \in \mathcal{M}_k(N, \chi)$, let $\alpha$ be a Dirichlet character modulo $N' \mid N$.  Then
\[S_{\alpha}(f) \in \mathcal{M}_k(NN', \chi \alpha^2).\] 
Let $q$ be any divisor of $N$ that is coprime to $N'$, then
\[S_\alpha(f) | U_q = \alpha(q) S_\alpha(f | U_q).\]
Similarly, if $S$ is a set of prime divisors of $N$ such that $N_S$ and $N'$ are coprime, then
\[S_\alpha(f) | W_S^{NN'} = \overline{\alpha}(N_S) S_\alpha(f | W_S^N).\]
\end{prop}

\section{A vanishing criterion for cusp forms}\label{sctn:eichler--shimura}

In this section we prove several criteria that relate the vanishing of twisted $L$-values of a modular form $f$ to the vanishing of $f$. We recall some facts from the theory of modular symbols; for details see \cite{Merel1994} or \cite{Stein2007} \S 8.  Let $k$ be an integer $\geq 2$.  The space $\mathbb{M}_k(\Gamma_1(N))$ of modular symbols is generated by the Manin symbols $[P, g]$ where $P$ is a homogeneous polynomial in $\C[X,Y]$ of degree $k-2$, and $g \in \SL_2(\Z)$.  In fact the Manin symbols $[P, g]$ only depends on $P$ and the coset $\Gamma_1(N)g$.  By mapping a matrix $g$ to its bottom row modulo $N$, the cosets of $\Gamma_1(N)\backslash\SL_2(\Z)$ are in bijection with the set
\[E_N = \{(u,v) \in (\Z/N\Z)^2;\: \gcd(u, v, N) = 1\};\]
see for example \cite{Stein2007} Proposition 8.6.  Note that $\gcd(u, v, N)$ is well-defined, i.e. does not depend on the choice of representative of the residue classes $u$ and $v$.  We write $[P, (u, v)] = [P, g]$ for any $g \in \SL_2(\Z)$ with bottom row congruent to $(u, v)$ modulo $N$.
For $f\in\mathcal{S}_k(\Gamma_1(N))$ let $\xi_f$ be the map
\[
\xi_f([P,g]) = \int_{g0}^{g\infty} f(z) (gP)(z,1)dz,
\]
with $g \in \SL_2(\Z)$ acting on $P\in\C[X, Y]$ by $(gP)(X,Y)= P(g^{-1}(X,Y)^T)$.\\

Using the generators $[X^j Y^{k-2-j}, (u, v)]$, where $0 \leq j \leq k-2$ and $(u, v) \in E_N$, we define $\xi_f(j; u, v) = \xi_f([X^j Y^{k-2-j}, (u, v)])$.  The Manin symbol relations from \cite[Theorem 8.4]{Stein2007} translate to
\begin{align}
\label{eq:manin-symb-rel1} &\xi_f(j; u, v) + (-1)^j \xi_f(k-2-j; v, -u) = 0,\\
\label{eq:manin-symb-rel2} &\xi_f(j; u, v) + \sum_{i=0}^{k-2-j}(-1)^{k-2-i} \binom{k-2-j}{i} \xi_f(i; v, -u-v) \nonumber \\
&\qquad + \sum_{i=k-2-j}^{k-2} (-1)^i \binom{j}{i-k+2+j}\xi_f(i; -u-v, u) = 0,\\
\label{eq:manin-symb-rel3} &\xi_f(j; u, v) - (-1)^{k-2} \xi_f(j, -u, -v)=0,
\end{align}
There is an involution $\iota$ of $\mathbb{M}_k(\Gamma_1(N))$, namely $\iota [X^j Y^{k-2-j}, (u, v)] = (-1)^{j+1}[X^j Y^{k-2-j}, (-u, v)]$.  Accordingly, we define
\[\xi_f^{\pm}(j; u, v) := \frac{\xi_f(j; u, v) \pm (-1)^{j+1} \xi_f(j; -u, v)}{2}.\]
The relations \eqref{eq:manin-symb-rel1}, \eqref{eq:manin-symb-rel2}, and \eqref{eq:manin-symb-rel3} hold because of a relation on the underlying Manin symbols, see \cite{Stein2007} \S 8.2.1.  One can apply $\iota$ to these relations for the Manin symbols to obtain another set of relations.  Applying $\xi_f$ and adding or subtracting as appropriate, we see that \eqref{eq:manin-symb-rel1}, \eqref{eq:manin-symb-rel2}, and \eqref{eq:manin-symb-rel3} hold for $\xi$ replaced by $\xi_f^{\pm}$.\\

By \cite{Merel1994} Proposition 8 the maps $f\mapsto \xi_f^+$ and $f\mapsto \xi_f^-$ are injective, so $f$ vanishes if all $\xi_f^{\pm}(j; u, v)$ do.  Note also that the $\xi_f(j; u, v)$ are related to critical values of $L$-functions: Indeed, taking $g = \left(\begin{smallmatrix} a & b \\ c & d \end{smallmatrix}\right)\in \SL_2(\Z)$ with $(c, v) \equiv (u, v) \bmod N$ we have 
\begin{equation}\label{eqn:xif-as-l-value}\xi_f(j; u, v) = \frac{j!}{(-2\pi i)^{j+1}} L(f|g, j+1),\end{equation}
where $L(f|g,j+1)$ is defined as follows. The $L$-series of a cusp form $F=\sum_{n\in\Q_{>0}} b_n e^{2\pi i n \tau}$ for a congruence subgroup is
\[
L(F,s) = \sum_{n\in\Q_{>0}} \frac{b_n}{n^s}.
\]
It converges for $\Re s\gg 0$ and can be extended analytically to the whole complex plane. We also denote this extended function, the $L$-function of $F$, by $L(F,s)$. The main goal of this section is to prove Theorem \ref{thm:eichler-shimura-triv-char}, which is a vanishing criterion for $f$ in terms of vanishing of certain twisted $L$-functions.  The result is in the spirit of Corollaire 2 of \cite{Merel2009}, although we require some modifications since we do not assume that $f$ is a newform, or even an eigenfunction of almost all Hecke operators.  First we recall an identity from the proof of Proposition 6 in \cite{Merel2009}:

\begin{lem}\label{lem:merel-coset-reps}  Let $N \in \N$, let $(u, v) \in E_N$, let $S$ denote the set of prime divisors of $N$ which divide $u$, let $\overline{S}$ denote the remaining prime divisors of $N$, and let $N'$ be the order of $uv$ in $\Z/N\Z$.  Let $g = \left(\begin{smallmatrix} a & b \\ c & d \end{smallmatrix}\right) \in \SL_2(\Z)$ be such that $(c, d) \equiv (u, v) \bmod N$.  Then
\[\Gamma_1(N)g = \Gamma_1(N) \begin{pmatrix} 0 & -1 \\ N & 0 \end{pmatrix} \begin{pmatrix} 1 & \frac{n}{N} \\ 0 & 1 \end{pmatrix} \begin{pmatrix} A & B \\ C & D \end{pmatrix} \begin{pmatrix} N N'_S & 0 \\ 0 & N_S\end{pmatrix}^{-1},\]
where $n$ is chosen so that $n \equiv uv \bmod N_{\overline{S}}$ and $n \equiv -uv \bmod N_S$, and $\left(\begin{smallmatrix} A & B \\ C & D \end{smallmatrix}\right) \in \Z^{2 \times 2}$ satisfies $AD - BC = N_S N'_S$, $A \equiv uN'_S \bmod N_{\overline{S}}$, $B \equiv v/N_{\overline{S}} \mod N_S$, and $N_S N'_S \mid A$, $N_S N'_S \mid D$, $N N' \mid C$, $N_{\overline{S}} N'_{\overline{S}} \mid B$.
\end{lem}

The existence of $A, B, C, D$ follows from the Chinese Remainder Theorem.  We omit the proof of this identity, which is simply a matter of checking that the matrix on the right hand side is integral with determinant one and with bottom row congruent to $(u, v)$ modulo $N$ (whence the conditions on $A, B, C, D$).\\

Let $N$ be a positive integer and let $T$ be the set of prime divisors $p$ of $N$ with $v_p(N)=1$, so
\[N=\prod_{p\in T} p \prod_{p\in\overline{T}} p^{v_p(N)},\]
where the exponents $v_p(N)$ for $p\in\overline{T}$ are all greater than $1$.
\begin{thm}\label{thm:eichler-shimura-triv-char}
Let $\epsilon\in\{0,1\}$, $N \in \N$, $k \geq 2$, and let $f \in \mathcal{S}_k^{\text{new}}(N)$ be an eigenfunction of all Atkin--Lehner operators $W_{p}^N$ for $p\in T$, with $T$ as above.  Assume that $L(f_{\alpha} | W_S^{NM}, j+1) = 0$ for all primitive characters $\alpha$ modulo $M \mid N$, all $j=0, 1, ..., k-2$ such that $\alpha(-1)=(-1)^{j+\epsilon}$, and all sets $S\subseteq \overline{T}$ of prime divisors $p$ that divide $\frac{N}{M}$.  Then $f=0$.
\end{thm}

\begin{proof}  We will present the argument for the case $\epsilon=1$, which uses the function $\xi^+$.  The other case, using $\xi^{-}$, is almost idential, the only difference being which characters cancel in \eqref{eq:period-calculation-pm}.  We will show that the conditions in the theorem imply $\xi^+_{f|W_N}(j; u, v) = 0$ for all $j = 0, 1, ..., k-2$ and $(u, v) \in E_N$, which in turn implies that $f=0$.  Let us therefore fix $(u, v) \in E_N$ and consider $\xi^+_{f|W_N}(j; u, v)$. As in the statement of Lemma \ref{lem:merel-coset-reps}, let $S$ be the set of those prime divisors of $N$ that divide $u$ and write $N'$ for the order of $uv$ in $\Z/N\Z$.   Note that every prime in $S$ divides $\frac{N}{N'}$. Choose $g = \left(\begin{smallmatrix} a & b \\ c & d \end{smallmatrix}\right) \in \SL_2(\Z)$ such that $(c, d) \equiv (u, v) \bmod N$.  By Lemma \ref{lem:merel-coset-reps} we have
\begin{equation}\label{eqn:Gamma_1(N)g-reps}
\Gamma_1(N)g = \Gamma_1(N)\begin{pmatrix}0&-1\\ N& 0\end{pmatrix}
\begin{pmatrix}1&\frac{n}{N}\\0&1\end{pmatrix}
\begin{pmatrix}A&B\\C&D\end{pmatrix}\begin{pmatrix}N N_S' & 0 \\ 0 & N_S\end{pmatrix}^{-1},
\end{equation}
with $A, B, C, D$ and $n$ satisfying the conditions of the lemma.  Since $f | W_N | W_N = f$, we have
\[ f|W_N|g = f | \begin{pmatrix} 1 & \frac{n}{N} \\ 0 & 1 \end{pmatrix} \begin{pmatrix}A&B\\C&D\end{pmatrix}\begin{pmatrix}NN_S' & 0 \\ 0 & N_S\end{pmatrix}^{-1}.\]
Now $n \equiv uv \bmod N_{\overline{S}}$ and $n \equiv -uv \bmod N_S$, so $n$ also has order $N'$ modulo $N$.  Hence $nN' = n' N$ for some $n'$ which is coprime to $N'$.  Writing this as $n/N = n'/N'$ and using \eqref{eqn:fourier-inversion} we get
\[f|W_N|g = \sum_{\alpha \bmod N'}\frac{\alpha(n')}{\phi(N')}S_\alpha(f)| \begin{pmatrix}A&B\\C&D\end{pmatrix}\begin{pmatrix}N N'_S & 0 \\ 0 & N_S\end{pmatrix}^{-1},
\]
where $\alpha$ varies over all Dirichlet characters modulo $N'$.  By Proposition \ref{prop:twists-commute-with-al-and-hecke} we have $S_\alpha(f)\in \mathcal{S}_k(NN',\alpha^2)$, and the conditions of Lemma \ref{lem:merel-coset-reps} together with Proposition \ref{prop:al} give
\[S_\alpha(f) | \begin{pmatrix} A & B \\ C & D \end{pmatrix} = \overline{\alpha^2_S}(B) \overline{\alpha^2_{\overline{S}}}\left(\frac{A}{N_S N'_S}\right) S_\alpha(f) | W_S^{NN'}.\]
Hence, using \eqref{eqn:xif-as-l-value},
\[\begin{aligned}
&\xi_{f|W_N}(j;u,v) \\
&\qquad =\frac{j!}{(-2\pi i)^{j+1}\phi(N')} \sum_{\alpha} \alpha(n')\overline{\alpha^2_S}(B) \overline{\alpha^2_{\overline{S}}}\left(\frac{A}{N_SN'_S}\right) L\left(S_\alpha(f)|W_S^{NN'}|\begin{pmatrix}\frac{1}{N N'_S }&0\\ 0&\frac{1}{N_S}\end{pmatrix},j+1\right)\\
\label{eq:period-calculation}
&\qquad = \frac{j! (N_{\overline{S}}N'_S)^{j+1-\frac{k}{2}}}{(-2\pi i)^{j+1}\phi(N')} \sum_{\alpha} \alpha(n')\overline{\alpha^2_S}(B) \overline{\alpha^2_{\overline{S}}}\left(\frac{A}{N_S N'_S}\right) L\left(S_\alpha (f)|W_S^{NN'},j+1\right),
\end{aligned}\]
where the sum is over all characters modulo $N'$. Here we used that for any cusp form $G$ of weight $k$ for a congruence subgroup, we have $L(G|\left(\begin{smallmatrix}t_1&0\\0&t_2\end{smallmatrix}\right),s) = (t_1/t_2)^{k/2-s}L(G,s)$.\\  

To compute $\xi_{f|W_N}(j;-u, v)$ we proceed analogously with $\tilde{g} = \left(\begin{smallmatrix} a & -b \\ -c & d \end{smallmatrix}\right)$, since this has bottom row $(-c, d) \equiv (-u, v) \bmod N$.  With $A, B, C, D, n$ as in \eqref{eqn:Gamma_1(N)g-reps} we see that 
\begin{equation}\label{eqn:Gamma_1(N)tildeg-reps}
\Gamma_1(N)\tilde{g} = \Gamma_1(N)\begin{pmatrix}0&-1\\ N& 0\end{pmatrix}
\begin{pmatrix}1&-\frac{n}{N}\\0&1\end{pmatrix}
\begin{pmatrix}-A&B\\C&-D\end{pmatrix}\begin{pmatrix} N N'_S & 0 \\ 0 & N_S \end{pmatrix}^{-1}.\end{equation}
The argument is as above, with $n'$ replaced by $-n'$, and each individual summand in the final expression for $\xi_{f|W_N}(j; u, v)$ changes by a factor of $\alpha(-1)\overline{\alpha_{\overline{S}}^2}(-1) = \alpha(-1)$.  From the definition of $\xi^{+}_{f|W_N}$ we then see
\begin{align}\label{eq:period-calculation-pm}
\xi^{+}_{f|W_N}(j;u,v)=\frac{j! (N_{\overline{S}}N'_S)^{j+1-\frac{k}{2}}}{(-2\pi i)^{j+1} \phi(N')} \sum_{\alpha} \alpha(n')\overline{\alpha^2_S}(B) \overline{\alpha^2_{\overline{S}}}\left(\frac{A}{N_S N'_S}\right)
L\left(S_\alpha(f)|W_S^{NN'},j+1\right).
\end{align}
where the sum is over all characters $\alpha$ modulo $N'$ with $\alpha(-1)= (-1)^{j+1}$.\\

The next step is to relate $S_\alpha(f)$ to the twist by the primitive character underlying $\alpha$.  The key to this is the following lemma, which can be proved by a direct computation:

\begin{lem}\label{lem:twist-by-non-prim-char}
Let $N$ and $k$ be positive integers, let $\chi$ be a Dirichlet character modulo $N$, and let $f \in \mathcal{S}_k(N, \chi)$.  Let $N' \in \N$, let $\alpha$ be a character modulo $N'$ with conductor $M$.  Assume that $M<N'$, let $p$ be any prime dividing $N'/M$, and let $\beta$ be the character modulo $N'/p$ inducing $\alpha$.  Then
\[ S_\alpha(f) = p^{1-k/2} S_\beta(f|U_p) | \begin{pmatrix} p& 0 \\ 0 & 1 \end{pmatrix} - \overline{\beta}(p) S_\beta(f).\]
\end{lem}

In our case $f \in \mathcal{S}_k^{\text{new}}(N)$ is an eigenfunction of each $W_p^N$ for $p\in T$, so by Proposition \ref{prop:al-hecke-connection} it is also an eigenfunction of $U_p$ for each $p \in T$. By the same proposition if $p\in\overline{T}$, i.e., $p^2|N$, then $U_p$ is the zero operator on $\mathcal{S}_k^{\text{new}}(N)$, so $f$ is trivially an eigenfunction of $U_p$ in that case. In both cases we denote the $U_p$-eigenvalue of $f$ by $a_p$.  Then Lemma \ref{lem:twist-by-non-prim-char} gives
\[S_{\alpha}(f) = p^{1-k/2} a_p S_\beta(f) | \begin{pmatrix} p & 0 \\ 0 & 1 \end{pmatrix} - \overline{\beta}(p) S_\beta(f),\]
and so
\[L(S_{\alpha}(f) | W_S^{NN'}, j+1) = (p^{-j} a_p - \overline{\beta}(p)) L(S_\beta(f) | W_S^{NN'}, j+1).\]
Applying this repeatedly we see that $L(S_\alpha(f) | W_S^{NN'}, j+1)$ is a multiple of $L(S_{\alpha_0}(f) | W_S^{NN'}, j+1)$, where $\alpha_0$ is the the primitive character modulo $M \mid N'$ inducing $\alpha$ modulo $N'$.  Next note that $S_{\alpha_0}(f)=G(\overline{\alpha_0})f_{\alpha_0} \in \mathcal{S}_k(NM, \alpha_0^2)$.  We then use $S_{\alpha_0}(f) | W_S^{NN'} = S_{\alpha_0}(f) | W_S^{NM} | B_d$, where $d = (\frac{N'}{M})_S$.  Thus $L(S_\alpha(f) | W_S^{NN'}, j+1)$ is a multiple of $L(f_{\alpha_0} | W_S^{NM}, j+1)$. If $S\subseteq \overline{T}$, then $L(f_{\alpha_0} | W_S^{NM}, j+1)=0$, as this is one of the $L$-values that is assumed to vanish in the statement of the theorem.\\

Now suppose that there is a prime $p\in S\cap T$. Since $v_p(N)=1$ and $p$ divides $u$, we have $\gcd(p,N')=\gcd(p,M)=1$. Write $S = S'\cup\{p\}$ and set $\alpha'=(\alpha_0)_{S'}$, the $S'$-part of $\alpha_0$. By Propositions \ref{prop:al} and \ref{prop:twists-commute-with-al-and-hecke} we have
\[
f_{\alpha_0} | W_S^{NM} =  \alpha'^2(p) (f_{\alpha_0}|W_{p}^{NM})|W_{S'}^{NM} = \alpha'^2(p)\overline{\alpha_0}(p)((f|W_p^{NM})_{\alpha_0})|W_{S'}^{NM}.
\]
Since we assume that $f$ is an eigenfunction of $W_p^N$ we get that $f_{\alpha_0} | W_S^{NM}$ is a multiple of $f_{\alpha_0} | W_{S'}^{NM}$.  Applying this procedure for every prime $p\in S\cap T$ we deduce that $f_\alpha|W_S^{NM}$ is a multiple of $f_{\alpha_0}|W_{S''}^{NM}$, where $S''$ is a set of prime divisors of $N/M$ that is disjoint with $T$. Thus $L(f_\alpha|W_S^{NM},j+1)=0$, since we assume that $L(f_{\alpha_0}|W_{S''}^{NM},j+1)$ vanishes.\\

Since $\xi^{+}_{f|W_N}(j;u,v)$ is a sum of such $L$-values this shows that $\xi^{+}_{f|W_N}(j;u,v)=0$ for all $(u,v)\in E_N$, and hence $f=0$.

\end{proof}

Note that in Theorem \ref{thm:eichler-shimura-triv-char} Atkin--Lehner operators are used in two different ways: first at $p \in T$ where we insist $f$ is an eigenfunction, second at $S \subset \overline{T}$ where we insist the $L$-values of Atkin--Lehner images of twists vanish.  For our applications we will restrict to levels of the form $N=p^a q^b N'$ with $N'$ squarefree; a simple trick then allows us to do away with the latter use:

\begin{thm}\label{thm:eichler-shimura-paqb}  Let $\epsilon\in\{0,1\}$ and $N = p^a q^b N'$ where $p$ and $q$ are distinct primes, $a, b \in \Z_{\geq 0}\setminus \{1\}$, and $N'$ is squarefree and coprime to $pq$. Let $k \geq 2$, and let $f \in \mathcal{S}_k^{\text{new}}(N)$ be an eigenfunction of all Atkin-Lehner operators $W_\ell^N$ for primes $\ell|N'$. Assume that $L(f_\alpha, j+1) = 0$ for all characters $\alpha$ primitive modulo $M \mid N$ and all $j=0,...,k-2$ with $\alpha(-1)=(-1)^{j+\epsilon}$. Then $f=0$.\end{thm}

\begin{proof}  We may assume that $a>1$ or $b>1$, in particular that the set $\overline{T} \subseteq\{p,q\} $ is non-empty, since otherwise this is just the statement of Theorem \ref{thm:eichler-shimura-triv-char}.  Let $(u,v)\in E_N$.  If no prime of $\overline{T}$ divides $u$ then the proof of Theorem \ref{thm:eichler-shimura-triv-char} shows that $\xi^+_{f|W_N}(j;u,v)$ is a linear combination of the $L$-values $L(f_{\alpha},j+1)$ so our assumptions give $\xi^+_{f|W_N}(j;u,v)=0$.  The same argument for $(v,-u)\in E_N$ shows that if no prime of $\overline{T}$ divides $v$ then $\xi^+_{f|W_N}(k-2-j;v,-u)=0$; we can then use the modular symbols relation \eqref{eq:manin-symb-rel1} to see that $\xi^+_{f|W_N}(j;u,v)=0$.\\

Now suppose that both $u$ and $v$ are divisible by a prime in $\overline{T}$. By the definition of $E_N$, it cannot be the case that the same prime divides both, so we are in the case when $a>1$ and $b>1$, and we may assume that $p$ divides $u$ and $q$ divides $v$.  Then the residue class $-u-v$ is divisible by neither $p$ nor $q$, so $\xi_f^{+}(i; v, -u-v) = \xi_f^+(i; -u-v, u)=0$ for all $0 \leq i \leq k-2$ by the above.  Hence using \eqref{eq:manin-symb-rel2} we obtain $\xi_f(j; u, v)=0$.\end{proof}

\begin{rmk}  More generally, an argument similar to the proof of Theorem \ref{thm:eichler-shimura-paqb} shows that one can restrict the sets $S \subset \overline{T}$ in Theorem \ref{thm:eichler-shimura-triv-char} to avoid two prescribed elements of $\overline{T}$.\end{rmk}

In the next section we will show that a cusp form $f$ in $\mathcal{S}^{\mathrm{new}}_k(N)$ is orthogonal to products of Eisenstein series if many twisted $L$-values of $f$ vanish. A technical difficulty arises in our application of Theorem \ref{thm:eichler-shimura-triv-char} when $k \geq 4$, where we cannot deduce that $L(f,2)$ and $L(f,k-2)$ vanish due to the fact that the weight two Eisenstein series with trivial character is not holomorphic. To this end we prove a result which states that the problematic cases are in fact already a consequence of the other assumptions:

\begin{prop}\label{prop:eichler-shimura-trivial-char}
Let $N \in \N$, $k \geq 4$ be even and $f\in \mathcal{S}^{\mathrm{new}}_k(N)$.  Assume that $L(f_\alpha, j+1)=0$ for all primitive characters $\alpha$ modulo $M \mid N$ where $M>1$, and all $j=0,\ldots,k-2$ with $\alpha(-1)=(-1)^{j+1}$.  Assume moreover that $L(f, j+1) = 0$ for all odd $j$ with $3\leq j\leq k-5$. Then $L(f, 2)=0$ and $L(f, k-2) = 0$ must hold as well.\end{prop}

\begin{proof}  As in the proof of Theorem \ref{thm:eichler-shimura-paqb} we see $\xi_{f|W_N}^+(j; u, v) = 0$ as long as $j \neq 1, k-3$ and at least one of $\gcd(u, N) = 1$ or $\gcd(v, N)=1$ holds.  First assume that $k \geq 6$.  Applying \eqref{eq:manin-symb-rel2} with $j=2$, $(v, -u-v) = (0, 1)$ and using the vanishing we just observed, we get
\[-(k-4) \xi_{f|W_N}^+(1; 0, 1) - 2 \xi_{f|W_N}^+(k-3; 1, 0) = 0.\]
Relation \eqref{eq:manin-symb-rel1} with $j=1$, $(u, v) = (0, 1)$ gives
\[\xi_{f|W_N}^+(1; 0, 1) - \xi_{f|W_N}^+(k-3; 1, 0) = 0,\]
hence
\[\xi_{f|W_N}^+(1; 0, 1) = 0,\]
since $k \geq 6$.  Since $\xi_{f|W_N}^+(1; 0, 1) = \xi_{f|W_N}(1; 0, 1)$, \eqref{eqn:xif-as-l-value} gives $L(f|W_N, 2) = 0$, hence $L(f, k-2)=0$ by the functional equation of $L(f,s)$.  The other case follows a similar argument, applying \eqref{eq:manin-symb-rel2} with $j=2$ and $(v, -u-v) = (1, 0)$ and \eqref{eq:manin-symb-rel1} with $j=1$ and $(u, v) = (1, 0)$ we get $\xi_{f|W_N}^+(k-3; 0, 1)=0$, hence $L(f, 2)=0$.\\

For $k=4$, apply \eqref{eq:manin-symb-rel2} with $j=0$, $(v, -u-v)=(0, 1)$ to get
\[\xi_{f|W_N}^{+}(1; 0, 1) = \frac{1}{(-2\pi i)^2}L(f,2)=
 0.\]
\end{proof}

\section{Generating spaces of cusp forms by products of Eisenstein series}\label{sctn:generating-cusp-forms}

We begin by recalling the theory of Eisenstein series as developed in \cite{Miyake2006} \S 7.  Let $N \in \N$, $l \in \N$, and let $\phi$ and $\psi$ be Dirichlet characters modulo $N_1$ and $N_2$ respectively such that $N_1N_2=N$.  We assume that $\phi(-1)\psi(-1) = (-1)^l$, and we extend $\phi$ to a character of $\Gamma_0(N)$ by $\phi\left(\left(\begin{smallmatrix} a & b \\ c & d \end{smallmatrix}\right)\right) = \phi(d)$, and similarly for $\psi$.  Define the Eisenstein series
\[E_l^{\phi,\psi}(z, s) = \frac{(l-1)!N_2^l}{(-2\pi i)^l G(\overline{\psi_0})}\sum_{\substack{(c,d)\in\Z^2\setminus\{(0,0)\}}}\frac{\phi(c)\overline{\psi}(d)}{(Ncz+d)^l|Ncz+d|^{2s}},
\]
where $\psi_0$ is the primitive character that induces $\psi$. It converges uniformly and absolutely for $l + 2\Re(s) \geq 2+\epsilon$, for any $\epsilon>0$.
In the region of absolute convergence this satisfies the transformation law 
\begin{equation}\label{eqn:analytic-eis-trans}E_l^{\phi,\psi}(\delta z, s) = \phi(\delta)\psi(\delta) j(\delta, z)^l \abs{j(\delta, z)}^{2s} E_l^{\phi,\psi}(z, s)\end{equation}
for $\delta \in \Gamma_0(N)$. The function $E_l^{\phi, \psi}(z, s)$ can be analytically continued in the $s$-variable to $s=0$ (see \cite[\S 7]{Miyake2006}), so we can define $E_{l}^{\phi,\psi}(z) = E_l^{\phi,\psi}(z, 0)$.  Moreover, unless $l=2$ and $\psi$ is principal, the value at $s=0$ is a holomorphic function of $z$, so \eqref{eqn:analytic-eis-trans} along with appropriate growth estimates shows that in fact $E_l^{\phi,\psi} \in \mathcal{M}_l(N, \phi\psi)$.\\

If $\phi$ and $\psi$ are primitive, the Fourier expansion of $E_l^{\phi,\psi}$ can be deduced from Theorems 7.1.3, 7.2.12, and 7.2.13 of \cite{Miyake2006}:
\begin{align}\label{eqn:eis-fourier-coeff}
E_{l}^{\phi,\psi}(z)=e_l^{\phi,\psi} + 2\sum_{n\geq 1}
\sigma_{l-1,\phi,\psi}(n)q^n\in\mathcal{M}_k(N,\phi\psi)
\end{align}
where  $\sigma_{l-1,\phi,\psi}(n) = \sum_{d|n}\phi(n/d)\psi(d)d^{l-1}$ and
\[
e_l^{\phi,\psi} = \begin{cases} 
L(\psi,1-l)&N_1=1,\\
L(\phi,0) &N_2=1\text{ and }l=1,\\
0 &\text{else.}
\end{cases}
\]
In the special case $\phi=\textbf{1}$ the Eisenstein series $E_l^{\textbf{1},\psi}(z,s)$, appropriately normalised, is given by a Poincar\'{e} series:
\begin{align}
\frac{2(-2\pi i)^lL(\overline{\psi},l+2s)G(\overline{\psi_0})}{(l-1)!N^l}E_l^{\textbf{1},\psi}(z, s)=\sum_{\gamma \in \Gamma_{\infty} \backslash \Gamma_0(N)} \frac{\overline{\psi(\gamma)}}{j(\gamma, z)^l \abs{j(\gamma, z)}^{2s}}.
\end{align}
\\

Let $k \in \N$, $\chi$ be a Dirichlet character modulo $N$ with $\chi(-1)=(-1)^{k}$, and let $f \in \mathcal{S}_{k}(N, \chi)$.  Given any $g \in \mathcal{M}_l(N, \overline{\psi} \chi)$, we consider the inner product
\[\langle gE_{k-l}^{\textbf{1},\psi}(\cdot,s), f \rangle = \int_{\Gamma_0(N) \backslash \mathcal{H}} g(z) E_{k-l}^{\textbf{1},\psi}(z,s) \overline{f(z)} y^{s + k} d\mu(z).\]
Here $d\mu(z) = (dx dy)/y^2$ is the $\SL_2(\R)$-invariant measure on the upper half plane.  Note that the integrand is indeed $\Gamma_0(N)$-invariant so the integral over this quotient is well-defined, at least when it converges.  This is the case if $s$ has sufficiently large real part, which we assume for the following proposition:

\begin{prop}\label{prop:inner-prod-as-l-values}  Let $N, k, l \in \N$, $\chi$ be a Dirichlet character modulo $N$, and $f$ be a newform in $\mathcal{S}_{k}(N, \chi)$. Let $\phi,\psi$ be Dirichlet characters modulo $N$ such that $\phi\psi=\chi$, $\phi(-1)=(-1)^l$ and denote by $\phi_0$ the primitive character that induces $\phi$.  Exclude the two cases $\phi_0 = \mathbf{1}$ and $l=2$, and $\phi=\chi$ and $l=k-2$.  Then
\begin{multline}\label{eq:inner-product-as-l-values}
 \langle E_{l}^{\textbf{1},\phi_0} E_{k-l}^{\textbf{1},\psi}(\cdot,s), f \rangle =\\
 \frac{i^{k-l}\Gamma(s + k - 1)(k-l-1)!N^{k-l}}{2^{2s+3k-l-2} \pi^{s + 2k-l- 1}L(\overline{\psi},k-l+2s)^2G(\overline{\psi_0}) } L(f^c, s+k-1) L((f^c)_{\phi_0}, s+k-l),
\end{multline}
where $f^c(z) = \overline{f(-\overline{z})}\in \mathcal{S}_k(N, \overline{\chi})$.
\end{prop}
\begin{proof}  Using the Rankin-Selberg method (see \cite{Shimura1976}) we obtain
\begin{equation}\label{eqn:rankin--selberg-unfolded}\langle gE_{k-l}^{\textbf{1},\psi}(\cdot, s), f \rangle = \frac{\Gamma(s + k - 1)(k-l-1)!N^{k-l}}{2(4 \pi)^{s + k- 1}(-2\pi i)^{k-l}L(\overline{\psi},k-l+2s)G(\overline{\psi_0}) } \sum_{n \geq 1} \frac{\overline{a_n} b_n}{n^{s + k  -1}},\end{equation}
where $a_n$ and $b_n$ are the Fourier coefficients of $f$ and $g$.  Note that $\overline{a_n}$ are the Fourier coefficients of $f^c(z)$.  A standard computation (see e.g. \cite{Raum2016} Proposition 4.1\begin{footnote}{Our divisor function is $\sigma_{l-1, \phi, \mathbf{1}}$ in Raum's notation.}\end{footnote}) gives
\[\sum_{n \geq 1} \frac{\overline{a_n} \sigma_{l-1, \mathbf{1}, \phi_0}(n)}{n^{s+k-1}} = \frac{L(f^c, s+k-1) L((f^c)_{\phi_0}, s+k-l)}{L(\overline{\chi}\phi_0, 2s + k-l)} .\]
So taking $g = E_l^{\mathbf{1}, \phi}$ in \eqref{eqn:rankin--selberg-unfolded} and using \eqref{eqn:eis-fourier-coeff} we obtain the result.  \end{proof}

Note that both sides of \eqref{eq:inner-product-as-l-values} have analytic continuation to $\C$, so by the uniqueness of analytic continuation the equality also holds at $s=0$:
\begin{cor}\label{cor:inner-prod-as-l-values}  Under the hypotheses of Proposition \ref{prop:inner-prod-as-l-values},
\[\begin{aligned} &\langle E_{l}^{\textbf{1},\phi_0} E_{k-l}^{\textbf{1},\psi}, f \rangle = \frac{i^{k-l}(k-l-1)!(k-2)!N^{k-l}}{2^{3k-l-2}\pi^{2k-l-1}L(\overline{\psi},k-l)^2G(\overline{\psi_0})} L(f^c, k-1) L((f^c)_{\phi_0}, k-l). \end{aligned}\]
\end{cor}

We can now proceed to the first result on generating cusp forms by products of Eisenstein series.

\begin{dfn}\label{dfn:PkN}  Let $N=p^aq^bN'$ be as in Theorem \ref{thm:eichler-shimura-paqb}.  For each $M \mid N$, write $D(M)$ for the set of primitive Dirichlet characters modulo $M$.  Let $B(N) \subset \bigsqcup_{M \mid N} D(M) \times \{1,...,k-1\}$ consist of the pairs $(\alpha, l)$ such that
\[\begin{aligned} \alpha(-1) &= (-1)^l, \\
(\alpha, l) &\neq (\mathbf{1}, 2), (\mathbf{1}, k-2). \end{aligned}\]
Define $P_k(N) \subset \mathcal{M}_k(N)$ to be the space generated by the products
\[
( E_l^{\textbf{1},\alpha} E_{k-l}^{\textbf{1},\overline{\alpha_N}}) | W_{S}^N
\]
for all $(\alpha,l) \in B(N)$ and all sets $S$ of prime divisors of $N'$.  Here $\alpha_N$ denotes the extension of $\alpha$ to a character modulo $N$.  \end{dfn}
\begin{thm}\label{thm:prods}  Let $N = p^aq^bN'$ be as in Theorem \ref{thm:eichler-shimura-paqb}.  For $M \subset \mathcal{M}_k(N)$, write $\overline{M}$ for the projection of $M$ to $\mathcal{S}_k^{\mathrm{new}}(N)$.  Then for $k \geq 4$ even
\[\overline{P_k(N)}=\mathcal{S}^{\mathrm{new}}_k(N).\]
In the case $k=2$ we define $\mathcal{S}^{\mathrm{new}}_{2,\mathrm{rk}=0}(N) \subset \mathcal{S}^{\mathrm{new}}_2(N)$ to be the subspace generated by newforms $f$ with $L(f,1)\neq 0$. Then
\[\overline{P_2(N)}=\mathcal{S}^{\mathrm{new}}_{2,\mathrm{rk}=0}(N).\]\end{thm}

\begin{proof} First assume $k>2$. As in Section \ref{sctn:eichler--shimura} denote the set of prime divisors of $N'$ by $T$. Assume that $\overline{P_k(N)}$ is a proper subset of $\mathcal{S}_k^{\mathrm{new}}(N)$. Since $P_k(N)$ is closed under the action of the Atkin-Lehner operators $W_S^N$ for $S\subseteq T$, so is the orthogonal complement of $\overline{P_k(N)}$ in $\mathcal{S}_k^{\mathrm{new}}(N)$. Therefore there exists a non-zero form $g\in \mathcal{S}^{\mathrm{new}}_k(N)$ that is orthogonal to $P_k(N)$ and an eigenform of the $W_S^N$.  We can write
\begin{equation}\label{eqn:g-as-sum-of-newforms}
g=\sum_{i=1}^r \beta_i f_i,
\end{equation}
where $f_1,\ldots,f_r$ are the newforms in $\mathcal{S}^{\mathrm{new}}_k(N)$ with the same $W_S^N$-eigenvalues as $g$ for all $S\subseteq T$. Using Corollary \ref{cor:inner-prod-as-l-values} (note $f_i = f_i^c$ since the $f_i$ have real Fourier coefficients) we get
\begin{align*}
\langle E_l^{\mathbf{1}, \alpha} E_{k-l}^{\overline{\mathbf{1}, \alpha_N}}, f_i\rangle =
 \frac{(k-l-1)!(k-2)!N^{k-l}}{(-2\pi i)^{k-l}(4\pi)^{k-1}L(\alpha_N,k-l)^2 G(\alpha)} L(f_i, k-1) L((f_i)_{\alpha}, k-l).\end{align*}
Since $g$ is a $W_S^N$-eigenform for $S\subseteq T$ and the operators $W_S^N$ are self-adjoint, for all $S \subset T$ we have $\langle E_l^{\mathbf{1}, \alpha} E_{k-l}^{\overline{\mathbf{1}, \alpha_N}}|W_S^N, g\rangle=0$ if and only if $\langle E_l^{\mathbf{1}, \alpha} E_{k-l}^{\overline{\mathbf{1}, \alpha_N}}, g\rangle = 0$.  We see that orthogonality of $g$ to $P_k(N)$ is equivalent to
\begin{align}\label{eqn:orthogonality-of-g1}
\sum_{i=1}^r \beta_i L(f_i, k-1)L((f_i)_{\alpha}, k-l)=0
\end{align}
for all $(\alpha,l) \in B(N)$.  Following an idea from the proof of Theorem 1 \cite[Theorem 1]{KohnenMartin2008}, we define another form in $G \in \mathcal{S}_k^{\mathrm{new}}(N)$ by
\[G = \sum_{i=1}^r \beta_i L(f_i,k-1)f_i.\]
Since the $f_i$ all have the same $W_S^N$-eigenvalues as $g$ for $S\subseteq T$, so does $G$. Then \eqref{eqn:orthogonality-of-g1} translates to
\begin{equation}\label{eqn:orthogonality-of-G} L(G_\alpha, k-l)=0\end{equation}
for $(\alpha, l) \in B$. Now applying Proposition \ref{prop:eichler-shimura-trivial-char} we see that $L(G, 2) = 0$ and $L(G, k-2) = 0$.  Thus $G$ satisfies the conditions of Theorem \ref{thm:eichler-shimura-paqb}, so $G = 0$.  Since $k \geq 4$, $L(f_i, k-1) \neq 0$, so all $\beta_i$ must be zero, and we arrive at the contradiction $g=0$.\\  

In the case where $k=2$ the proof is similar.  The inclusion $\overline{P_2(N)} \subset \mathcal{S}^{\mathrm{new}}_{2,\mathrm{rk}=0}(N)$ follows from Corollary \ref{cor:inner-prod-as-l-values}, which shows that $P_2(N)$ is orthogonal to every newform $f$ with $L(f,1)=0$. The rest of the argument works as above. \end{proof}

When $a = b = 0$, so $N = N'$ is squarefree, one easily sees that the set $B(N)$ in Definition \ref{dfn:PkN} satisfies
\[\#B(N) \sim_k \frac{k-1}{2} \prod_{p \mid N} (p-1)\]
for $k\to\infty$. Therefore, writing $p_k(N)$ for the number of generators for $P_k(N)$ as given in Definition \ref{dfn:PkN}, we have
\[p_k(N) \sim_k \frac{k-1}{2} \prod_{p \mid N} 2(p-1).\]
This should be compared to
\[\dim \mathcal{S}_k^{\mathrm{new}}(N) \sim_k \frac{k-1}{12} \prod_{p \mid N} (p-1)\]
(c.f. \cite{Martin2005}).  It is therefore an interesting question whether we can remove the Atkin--Lehner operators appearing in the definition of $P_k(N)$ so to obtain a space $P_k(N)$ for which the number of generators is of similar size to the dimension of the target space.  When $N=p^a q^b$ for $a,b>1$ we have
\[p_k(N) \sim_k \frac{k-1}{2} p^{a-1}(p-1) q^{b-1}(q-1),\]
while (with $a, b \geq 3$ for simplicity),
\[\dim \mathcal{S}_k^{\mathrm{new}}(N) \sim_k \frac{k-1}{12} p^{a-1}(p-1)\left(1 - \frac{1}{p^2}\right) q^{b-1}(q-1) \left(1-\frac{1}{q^2}\right).\]
In this case (as well as the case of $N=N'$) it is also an interesting question whether one can quantify the linear dependency among the products of Eisenstein series generating $P_k(N)$ (or $Q_k(N)$, defined below).  These questions may be asked for any $k$, but they are particularly pertinent when $k=2$, as any additional symmetry in this case which would allow lowering the constant further could have application to Brumer's conjecture (see \cite{Iwaniec2000}) regarding the number of newforms $f$ with $L(f,1) = 0$.\\

Most of the methods we have developed also work for the spaces $\mathcal{M}_k(N,\chi)$ where $\chi$ is a non-principal character modulo $N$.  However, there are some complications, in particular because the Atkin--Lehner operators $W_S^N$ are no longer endomorphisms of $\mathcal{M}_k(N,\chi)$ when $\chi$ is not quadratic. This means that it is necessary to take Eisenstein series coming from different eigenspaces of the diamond operators to generate $\mathcal{S}_k(N, \chi)$.  To minimise the technicalities we give an example of how the same methods can be used to treat the case of prime level, where the Atkin--Lehner operators are not needed:

\begin{thm}\label{thm:prods-gen-weight-2-prime}
Let $p$ be prime and let $\chi$ a character modulo $p$.  Let $P_2(p,\chi)$ be the space generated by 
\[E_{1}^{\textbf{1},\overline{\alpha}} E_{1}^{\textbf{1},\overline{\chi}\alpha},\]
for $\alpha$ varying over all (primitive) odd characters modulo $p$.  Write $\overline{P_2(p,\chi)}$ for the projection of $P_2(p, \chi)$ to $\mathcal{S}_2(p,\chi)$.  Let $\mathcal{S}^{\mathrm{new}}_{2,\mathrm{rk}=0}(p, \chi) \subset \mathcal{S}_2(p, \chi)$ be the subspace generated by newforms $f$ with $L(f, 1) \neq 0$.  Then
\[
\overline{P_2(p,\chi)}=\mathcal{S}^{\mathrm{new}}_{2,\mathrm{rk}=0}(p, \chi)
\]
\end{thm}

\begin{proof}
The proof proceeds along the same lines as the proof of \ref{thm:prods}. Instead of Theorem \ref{thm:eichler-shimura-paqb} we use the following vanishing criterion, which can be proved with the same methods as Theorem \ref{thm:eichler-shimura-triv-char}. If, for $G\in\mathcal{S}_2(p,\overline{\chi})$, the $L$-value $L(G_\alpha,1)$ for all odd characters $\alpha$ modulo $p$ vanishes, then $G$ vanishes.
\end{proof}

\section{The new part of $P_k(N)$}\label{sctn:new-part}

In this section we will analyse the new parts of the generators of $P_k(N)$ for any $N$. We use this to construct another space $Q_k(N)$ with the same projection to the new space as $P_k(N)$ whose generators do not involve Atkin--Lehner operators. While $P_k(N)$ was more useful for the proof of Theorem \ref{thm:prods}, $Q_k(N)$ is more explicit and easy to implement on a computer.  The first step is to write $E_{k-l}^{\textbf{1},\overline{\alpha_N}}$ in terms of the Eisenstein series attached to the underlying primitive character:

\begin{lem}\label{lem:eisenstein-decomposition} Let $\alpha$ be a primitive character modulo $M$ with $\alpha(-1)=(-1)^k$.  Writing $N = \prod p^{v_p(N)}$, let $N_M = \prod_{p \mid M} p^{v_p(N)}$ be the $M$-part of $N$, so that $M\mid N_M$ and $\gcd(M,N/N_M)=1$. Then
\[
E_{k-l}^{\textbf{1},\overline{\alpha_N}} = \left(\frac{N}{M}\right)^{\frac{k}{2}-l}\sum_{e\mid N/N_M}\mu(e)\alpha(e)e^{-\frac{k}{2}+l}E_{k-l}^{\textbf{1},\overline{\alpha}}|B_{N/Me}
\]
\end{lem}
\begin{proof}
The proof is analogous to the proof of \cite[Lemma 8.4.2]{CohenStromberg2017}. For Re$(s)\gg 0$ we have
\begin{align*}
E_{k-l,N}^{\mathbf{1}, \overline{\alpha_N}}(z,s) &= \frac{(k-l-1)!N^{k-l}}{(-2\pi i)^{k-l} G(\overline{\alpha})}\sum_{(c,d)\neq(0,0)} \frac{\alpha_N(d)}{(cNz+d)^{k-l}|cNz+d|^{2s}}.
\end{align*}
Using the fact that $\sum_{d \mid n} \mu(d)$ is the indicator function for $n=1$, we get
\begin{align*}
\sum_{(c,d)\neq(0,0)} \frac{\alpha_N(d)}{(cNz+d)^{k-l}|cNz+d|^{2s}}&=\sum_{(c,d)\neq(0,0)}\sum_{e\mid\gcd(d,N/N_M)}\mu(e)\frac{\alpha(d)}{(cNz+d)^{k-l}|cNz+d|^{2s}}\\
&=\sum_{e\mid N/N_M}\frac{\mu(e)\alpha(e)}{e^{k-l+2s}}\sum_{(c,d)\neq(0,0)}\frac{\alpha(d)}{(cM(\frac{N}{Me})z+d)^{k-l}|cM(\frac{N}{Me})z+d|^{2s}}\\
&= \frac{(-2\pi i)^{k-l} G(\overline{\alpha})}{(k-l-1)!M^{k-l}}\sum_{e\mid N/N_M}\mu(e)\alpha(e)e^{-k+l-2s}E_{k-l}^{\textbf{1},\overline{\alpha}}((N/Me)z,s).
\end{align*}
We obtain an equality of functions of the variable $s$, which remains true for $s=0$ by uniqueness of analytic continuation.\end{proof}

Thus the product $E_l^{\textbf{1},\alpha}E_{k-l}^{\textbf{1},\overline{\alpha_N}}$ is a linear combination of products of the form 
\[E_l^{\textbf{1},\alpha}\cdot \left(E_{k-l}^{\textbf{1},\overline{\alpha}}|B_{N/Me}\right)\]
for $e\mid N/N_M$. If $e\neq 1$ these products clearly have level smaller than $N$, so are old forms.  Hence the projection of $P_k(N)$ to the new space $\overline{P_k(N)}$ is generated by the projections of the products
\begin{align}\label{eqn:prod-generators}
\left(E_l^{\textbf{1},\alpha}|W_S^N\right) \cdot\left(E_{k-l}^{\textbf{1},\overline{\alpha}}|B_{N/M}|W_S^N\right).
\end{align}
where $S\subseteq T$ is a set of prime divisors of the squarefree part of $N$.  Let us focus on the first factor for now.  It is easy to see that, as operators on $\mathcal{M}_k(M, \alpha)$, we have the equality $W_S^N = W_{S_M}^M | B_{(N/M)_S}$, where $S_M$ is the set of primes in $S$ that divide $M$.

Using Proposition 14 of \cite{Weisinger1977} we see that the first factor in \eqref{eqn:prod-generators} is a multiple of
\[E_l^{\overline{\alpha}_{S_M},\alpha_{\overline{S}_M}} | B_{(N/M)_S},\]
where $\overline{S}_M = \{p \mid M\} \setminus S_M$.  To study the second factor in \eqref{eqn:prod-generators} we use an extension of Proposition 1.5 of \cite{AtkinLi78} that allows us to swap the order of the lifting operator and the Atkin--Lehner operator in equation \eqref{eqn:prod-generators}.
\begin{lem}\label{lem:al-lift-swap}
Let $F\in\mathcal{M}_k(M,\chi)$, $d\in\mathbb{Z}_{\geq 1}$, and $S$ be a set of primes dividing $dM$.  Let $\overline{S} = \{p \mid dM\} \setminus S$, $S_M = S \cap \{p \mid M\}$, and define $d_S = \prod_{p \in S} p^{v_p(d)}$ and $d_{\overline{S}}$ as usual.  Then
\begin{align*}
F | B_d | W_S^{Md}= \overline{\chi}_{S}(d_{\overline{S}}) \overline{\chi}_{\overline{S}}(d_S) F | W_{S_M}^M | B_{d_{\overline{S}}}
\end{align*}
\end{lem}
\begin{proof} Choose $x,y,z,w\in \Z$ as in the definition of $W_S^{Md}$, i.e. satisfying $y\equiv 1\pmod{d_SM_S}$, $x\equiv 1\pmod{d_{\overline{S}}M_{\overline{S}}}$ and $(M_Sd_S)^2xw-Mdzy=M_Sd_S$. As operators on $\mathcal{M}_k(N,\chi)$, we have
\[ B_dW_S^{Md}=\begin{pmatrix} d & 0 \\ 0 & 1 \end{pmatrix} \begin{pmatrix} d_SM_Sx & y \\ Md z & d_SM_Sw\end{pmatrix} 
= \begin{pmatrix} M_Sd_Sx & d_{\overline{S}}y \\ Mz & M_Sw \end{pmatrix}\begin{pmatrix} d & 0 \\ 0 & d_S \end{pmatrix}.\]
The determinant of $\left(\begin{smallmatrix} M_Sd_Sx & d_{\overline{S}}y \\ Mz & M_Sw \end{smallmatrix}\right)$ is $M_S$. The result now follows by Proposition \ref{prop:al} and the fact that $y\equiv 1 \pmod{M_S}$ and $x\equiv 1\pmod{M_{\overline{S}}}$. \end{proof}

Applying Lemma \ref{lem:al-lift-swap} with $d = N/M$ to $E_{k-l}^{\textbf{1},\overline{\alpha}}|B_{N/M} | W_S^N$ and using Proposition 14 of \cite{Weisinger1977}, we see that the second factor in \eqref{eqn:prod-generators} is a multiple of
\[
E_{k-l}^{\alpha_{S_M},\overline{\alpha}_{\overline{S}_M}}|B_{\left(\frac{N}{M}\right)_{\overline{S}}},
\]
so the product in \eqref{eqn:prod-generators} is a multiple of
\[
\left(E_l^{\overline{\alpha}_{S_M},\alpha_{\overline{S}_M}} | B_{\left(\frac{N}{M}\right)_{S}}\right) \cdot \left(E_{k-l}^{\alpha_{S_M},\overline{\alpha}_{\overline{S}_M}}|B_{\left(\frac{N}{M}\right)_{\overline{S}}}\right).
\]

In order to unwind this, set 
\[\begin{aligned} &M_1 = M_S= \prod_{p \in S_M} p^{v_p(M)},\: &&M_2 = M_{\overline{S}}= \prod_{p \in \overline{S}_M} p^{v_p(M)}, \\
&d_1 = (N/M)_S  = \prod_{p \in S} p^{v_p(N) - v_p(M)},\: &&d_2 = (N/M)_{\overline{S}}  = \prod_{p \in \overline{S}} p^{v_p(N) - v_p(M)}.\end{aligned}\]  
Note that $\overline{S}_M = \{p \mid M\} \setminus S_M \subset \overline{S}$.  With these definitions $\overline{\alpha}_{S_M}$ and $\alpha_{\overline{S}_{M}}$ are primitive characters modulo $M_1$ and $M_2$ respectively, which we now relabel as $\phi$ and $\psi$.  Note that as $\alpha$ varies over all primitive characters of parity $\epsilon$ modulo $M$, $\phi$ and $\psi$ vary over all primitive characters modulo $M_1$ and $M_2$ such that $\phi\psi$ has parity $\epsilon$.  Now fix $M|N$ and let $S\subseteq T$ vary: we obtain all $M_1, M_2$ such that $M_1\mid N_T$ and $M_1M_2 = M$, and for given $M_1, M_2$ we obtain all $d_1, d_2$ such that $d_1M_1\mid N_T$ and $d_1 M_1 d_2 M_2 = N$.

\begin{dfn}\label{dfn:QkN}  Let $N \in \N$.  Let $T$ be the set of primes $p$ such that $v_p(N)=1$.  Let $B'(N)$ consist of all quintuples $(\phi, \psi, l, d_1, d_2)$, where $\phi$ varies over all primitive characters of conductor $M_1$, with $M_1$ varying over all divisors of $N_T$, $\psi$ varies over all primitive characters of conductor $M_2$, with $M_2$ varying over all divisors of $N$, all $l \in \{1,\ldots,k-1\}$, and all $d_1,d_2\in\N$, such that
\begin{footnote}{Of course the second condition is only relevant when $M = M_1 = M_2 = 1$.}\end{footnote}
\[\begin{aligned} \phi\psi(-1) &= (-1)^l, \\
(\phi, \psi, l) &\neq (\mathbf{1}, \mathbf{1}, 2), (\mathbf{1}, \mathbf{1}, k-2), \\
d_1M_1&\mid N_T,\\
d_1 M_1 d_2 M_2 &= N.\end{aligned}\]
We define $Q_k(N)$ to be the vector space generated by
\begin{equation}\label{eq:final-form-of-products}
E_{l}^{\phi,\psi}|B_{d_1}\cdot E_{k-l}^{\overline{\phi},\overline{\psi}}|B_{d_2}.
\end{equation}
for all $(\phi, \psi, l, d_1, d_2) \in B'(N)$.
\end{dfn}  

The above calculation shows that $Q_k(N)$ and $P_k(N)$ have the same projection onto the new subspace $\mathcal{S}_k^{\mathrm{new}}(N)$.  Using the spaces $Q_k(N)$ and their lifts we can extend Theorem \ref{thm:prods} to the full space $\mathcal{S}_k(N)$: 

\begin{thm}\label{thm:prods-full-space}
Let $N$ be as in Theorem \ref{thm:eichler-shimura-paqb} and $\mathcal{Q}_k(N)=\bigcup\limits_{N_0 d | N}Q_k(N_0)|B_d$ be the subspace of $\mathcal{M}_k(N)$ generated by the products
\[
E_{l}^{\phi,\psi}|B_{d_1d}\cdot E_{k-l}^{\overline{\phi},\overline{\psi}}|B_{d_2d}
\]
for $(\phi, \psi, l, d_1,d_2) \in B'(N_0)$ (as in Definition \ref{dfn:QkN}).  Then for $k\geq 4$
\[
\mathcal{M}_k(N) =  \mathcal{Q}_k(N)+\mathcal{E}_k(N).
\]
\end{thm}
\begin{proof}
This follows from Theorem \ref{thm:prods}, the previous calculations, and an inductive argument using the fact that
\[\mathcal{S}_k(N) = \bigoplus_{N_0 | N}\bigoplus_{d\mid N/N_0}\mathcal{S}^{\mathrm{new}}_k(N_0)|B_d.\]\end{proof}

To treat the case $k=2$ we need one more result.

\begin{prop}\label{prop:oldforms-orthogonal-to-prods}
Let $f\in\mathcal{S}_2^{\mathrm{new}}(N_0)$ be a newform of level $N_0\mid N$ with $L(f,1)=0$, and let $d$ be such that $dN_0\mid N$. Then $f| B_d$ is orthogonal to $P_2(N)$.
\end{prop}
\begin{proof}
It suffices to show that $f| B_d$ is orthogonal to each of the generators of $P_2(N)$, so we fix a product
$
(E_1^{\textbf{1},\alpha} E_1^{\mathbf{1}, \overline{\alpha_N}}) | W_S^N
$
where $\alpha$ is a primitive odd character modulo $M \mid N$ and $S \subset T$ is a subset of the primes $p$ with $v_p(N)=1$. Since $W_S^N$ is self-adjoint,
\[
\langle (E_1^{\textbf{1},\alpha} E_1^{\mathbf{1}, \overline{\alpha_N}}) | W_S^N,f|B_d\rangle=
\langle E_1^{\textbf{1},\alpha} E_1^{\mathbf{1}, \overline{\alpha_N}},f|B_d | W_S^N).
\]
Using Lemma \ref{lem:al-lift-swap} and the fact that $f$ is an eigenfunction of all $W_{S'}^M$ for sets $S'\subseteq T$ of prime divisors of $N_0$, we see that $f|B_d | W_S^N$ is a multiple of $f|B_{d'}$ for some $d'|d$.  Arguing as in Proposition \ref{prop:inner-prod-as-l-values}, 
\[
\langle E_1^{\textbf{1},\alpha} E_1^{\mathbf{1}, \overline{\alpha_N}}(\cdot, s), f|B_{d'} \rangle = \frac{\Gamma(s + 1)}{d'^{s+1}(4 \pi)^{s + 1} } \sum_{n \geq 1} \frac{a_n \sigma_{1,\textbf{1},\alpha}(d'n)}{n^{s + 1}},
\]
where $a_n$ are the Fourier coefficients of $f$ (note that $f^c = f$ in this case), and $\Re ~ s \gg 0$.  Let $d'=\prod p^{e_p}$. Then
\begin{align}
\sum_{n \geq 1} \frac{a_n \sigma_{1,\textbf{1},\alpha}(d'n)}{n^{s + 1}}=\sum_{\gcd(n,d')=1}\frac{a_n \sigma_{1,\textbf{1},\alpha}(n)}{n^{s + 1}}\prod_{p\mid d'}\left(\sum_{b=0}^{\infty}\frac{a_{p^b} \sigma_{1,\textbf{1},\alpha}(p^{b+e_p})}{(p^{b})^{s + 1}}\right).
\end{align}

The first sum over $n$ coprime to $d'$ is, up to the Euler factors corresponding to the prime divisors of $d'$, given in the proof of Proposition \ref{prop:inner-prod-as-l-values}:
 \[\sum_{(n,d')=1} \frac{a_n \sigma_{0, \mathbf{1}, \alpha}(n)}{n^{s+1}} = \frac{L(f, s+1) L(f_{\alpha}, s+1)}{L(\alpha, 2s + 1)}\prod_{p|d'} \frac{ L_p(\alpha, 2s + 1)}{L_p(f, s+1) L_p(f_{\alpha}, s+1)}
,\]
where the Euler factors are $L_p(\alpha,2s+1) =(1-\alpha(p)p^{-(2s+1)})^{-1} $, $L_p(f,s+1)=(1-a_p p^{-(s+1)}+p^{-1-2(s+1)})^{-1}$, and $L_p(f_{\alpha},s+1)=(1- \alpha(p)a_p p^{-(s+1)}+\alpha(p)^2 p^{-1-2(s+1)})^{-1}$. Noting that the coefficients $a_p$ are algebraic integers because $f$ is a newform, we see that all the Euler factors are holomorphic at $s=0$ and do not vanish there. The same is true for $L(\alpha,2s+1)$, as $\alpha$ is primitive and non-trivial. Since we assume $L(f,1)=0$ the sum $\sum_{(n,d')=1} a_n \sigma_{0, \mathbf{1}, \alpha}(n)n^{-(s+1)}$ vanishes at $s=0$.

It remains to show that the sums 
\[
f_p(s)=\sum_{b=0}^{\infty}\frac{a_{p^b} \sigma_{1,\textbf{1},\alpha}(p^{b+e_p}) }{(p^{b})^{s + 1}}=
\sum_{b=0}^{\infty}\frac{a_{p^b} \sigma_{1,\textbf{1},\alpha}(p^{b})}{(p^{b})^{s + 1}}+\sum_{b=0}^{\infty}\frac{a_{p^b} \alpha(p^b)p^b}{(p^{b})^{s+1}}(\alpha(p)p+\ldots+\alpha(p^{e_p})p^{e_p})
\]
can be analytically continued to $s=0$ for all $p$ dividing $d'$. The first sum corresponds to the Euler factors at $p$ of the quotient of $L$-functions given in Proposition \ref{prop:inner-prod-as-l-values} and hence has analytic continuation to $s=0$. The second sum equals
\[
(\alpha(p)p+\ldots+\alpha(p^{e_p})p^{e_p})L_p(f_\alpha,s),
\]
which can again be analytically continued to $s=0$.
\end{proof}
We now define a space $\mathcal{S}_{2,\mathrm{rk}=0}(N)$ that projects onto $\mathcal{S}^{\mathrm{new}}_{2,\mathrm{rk}=0}(N)$ from Theorem \ref{thm:prods}, by
\[\mathcal{S}_{2,\mathrm{rk}=0}(N) = \bigoplus_{N_0 | N}\bigoplus_{d\mid N/N_0}\mathcal{S}^{\mathrm{new}}_{2,\mathrm{rk}=0}(N_0)|B_d.\]
By Proposition \ref{prop:oldforms-orthogonal-to-prods} $P_2(N)$ is contained in $\mathcal{S}_{2,\mathrm{rk}=0}(N)$ and by Theorem \ref{thm:prods-full-space} the two spaces have the same projection to the new space of $\mathcal{S}_2(N)$. This projection is equal to the projection of $Q_2(N)$ and so we can again use induction to prove:

\begin{thm}\label{thm:prods-full-space-weight-2}
Let $N$ be as in Theorem \ref{thm:eichler-shimura-paqb} and let $\mathcal{Q}_2(N) = \cup_{N_0 d \mid N} Q_2(N_0) | B_d$ be the subspace of $\mathcal{M}_2(N)$ generated by the products
\[
E_{1}^{\phi,\psi}|B_{d_1d}\cdot E_{1}^{\overline{\phi},\overline{\psi}}|B_{d_2d}
\]
for all $(\phi, \psi,1,d_1, d_2) \in B'(N_0)$. Then
\[
\mathcal{S}_{2,\mathrm{rk}=0}(N)+\mathcal{E}_2(N) = \mathcal{Q}_2(N)+\mathcal{E}_2(N).
\]
\end{thm}
Let $B$ be a basis of $\mathcal{S}_2(N)$ consisting of newforms of the form $f_i(d z)$, where $f_i$ is a newform for level $M_i|N$ and $d$ a divisor of $N/M_i$. Let $f$ be a cusp form that is given as a linear combination $f(z) = \sum_i \sum_{d|N/M_i} a_{i,d} f_i(dz)$. If $N$ is as in \ref{thm:eichler-shimura-paqb}, then Theorem \ref{thm:prods-full-space-weight-2} states that $f$ can be written as a linear combination of products of Eisenstein series in $\mathcal{Q}_2(N)$ and Eisenstein series in $\mathcal{E}_2(N)$ if and only if $a_{i,d}=0$ for all $i$ with $L(f_i,1)=0$.

\section{Fourier expansions at arbitrary cusps}\label{scn:Fourier expansions}

Let $f\in \mathcal{M}_k(\Gamma_0(N))$ and $\alpha = \frac{a}{c}\in\Q\cup\{\infty\}$ with $(a,c)=1$. In this section we discuss how a representation of $f$ as a linear combination of products of Eisenstein series can be used to obtain a Fourier expansion of $f$ at $\alpha$. The method described here applies to modular forms for any congruence subgroup but we focus on $\Gamma_0(N)$ because our main theorem \ref{thm:prods-full-space} applies to $\Gamma_0(N)$. Choose a matrix $\gamma\in\SL_2(\Z)$ that maps $\infty$ to $\alpha$. The form $f|\gamma$ is invariant under $T^w$, where $w = \frac{N}{\gcd(c^2,N)}$ is the width of the cusp $\alpha$. Hence $f|\gamma$ has a Fourier expansion in $q_w = e^{\frac{2\pi i\tau}{w}}$:
\[
f|\gamma(\tau) = \sum_{n=0}^\infty a_n^\gamma(f)q_w^n.
\]
We will call this a \textit{Fourier expansion of $f$ at the cusp $\alpha$}. Note that the coefficients $a_n^\gamma(f)$ depend on the choice of $\gamma$. Indeed, if $\gamma\infty = \alpha$, then also $\gamma T^m \infty = \alpha$, and
$a_n^{\gamma T^m}(f) = \zeta_w^{nm}a_n^\gamma(f)$, where $\zeta_w$ is a primitive $w$-th root of unity.

Now assume we have a representation of $f$ as a linear combination of Eisenstein series and products of two Eisenstein series. By Theorem \ref{thm:prods-full-space} and Theorem \ref{thm:prods-full-space-weight-2}, for $N = p^aq^b N'$ for squarefree $N'$ we can find such an expansion with products of the form
$
E_{l}^{\phi,\psi}|B_{d_1d}\cdot E_{k-l}^{\overline{\phi},\overline{\psi}}|B_{d_2d},
$
as long as $f$ is in $\mathcal{S}_{2,\mathrm{rk}=0}(N)+\mathcal{E}_2(N)$ if $k=2$. Since the slash operator $|\gamma$ is linear and satisfies
\[
(E_{l}^{\phi,\psi}|B_{d_1d}\cdot E_{k-l}^{\overline{\phi},\overline{\psi}}|B_{d_2d})|\gamma =
E_{l}^{\phi,\psi}|B_{d_1d}\gamma\cdot E_{k-l}^{\overline{\phi},\overline{\psi}}|B_{d_2d}\gamma,
\]
the problem of finding the Fourier expansion of $f|\gamma$ reduces to the problem of finding the expansion of $E_k^{\phi,\psi}|B_d\gamma$ for primitive characters $\phi,\psi$ and $d\in\mathbb{N}$. This is discussed in the second chapter of \cite{Weisinger1977}, although there are several mistakes in the formulas. A corrected version was communicated to the authors by Henri Cohen \cite{Cohen2017unpublished}.

We have implemented the above method of finding Fourier expansions of modular forms at arbitrary cusps in Sage \cite{sage} and present several examples of Fourier expansions below. Our program is available at \cite{Github_mneururer}. In our examples we will calculate Atkin-Lehner eigenvalues and the expansions at cusps of the form $\frac{1}{d}$ for $d|N$, although the expansions at other cusps can be obtained by the same method.
A matrix that maps $\infty$ to $\frac{1}{d}$ is given by $\gamma_d=\sabcd{1}{0}{d}{1}$.
\begin{enumerate}
\item
Let $f_{49}$ be the unique newform of level $49$ and weight $2$. It is a linear combination of the products $E_1^{\textbf{1},\phi}E_1^{\textbf{1},\overline{\phi}}$ and $E_1^{\textbf{1},\phi^3}E_1^{\textbf{1},\overline{\phi}^3}$, where $\phi$ is the character modulo $49$ that maps $3$ to $\zeta_{42}$:
\begin{multline*}
f_{49} = \frac{1}{28}(-20 \zeta_{21}^{11} + 5 \zeta_{21}^{10} + \zeta_{21}^{9} - 19 \zeta_{21}^{8} + 10 \zeta_{21}^{7} + 9 \zeta_{21}^{6} - 11 \zeta_{21}^{5} - 5 \zeta_{21}^{4} + 15 \zeta_{21}^{3} - 9 \zeta_{21}^{2} + \zeta_{21} + 21
)E_1^{\textbf{1},\phi}E_1^{\textbf{1},\overline{\phi}}\\
+\frac{1}{28}(4 \zeta_{21}^{11} - 2 \zeta_{21}^{10} + \zeta_{21}^{9} + 3 \zeta_{21}^{8} - 4 \zeta_{21}^{7} + \zeta_{21}^{6} + 3 \zeta_{21}^{5} - 2 \zeta_{21}^{4} - 2 \zeta_{21}^{3} + 5 \zeta_{21}^{2} - 5 \zeta_{21} - 6
)E_1^{\textbf{1},\phi^3}E_1^{\textbf{1},\overline{\phi}^3}.
\end{multline*} 
Using this we obtain the Fourier expansion of $f_{49}$ at $\alpha=0$.
\begin{equation}\label{eqn:f49 at 0}
f_{49}|S = \frac{1}{49}(-q_{49} - q_{49}^2 + q_{49}^4 + O(q_{49}^8)),
\end{equation}
We deduce that $f_{49}$ has $W_{49}$-eigenvalue $-1$. Indeed \eqref{eqn:f49 at 0} implies
\[
f_{49}|W_{49}(z) = f_{49}|SB_{49}(z) = 49 f_{49}|S(49z) = -q + O(q^2) = -f(z).
\]
We can see this already from the fact that $f_{49}$ has an expansion in terms of Eisenstein series and products of two Eisenstein series. According to Theorem \ref{thm:prods-full-space-weight-2} that implies that $L(f_{49},1)\neq 0$. Denoting the $W_{49}$-eigenvalue of $f_{49}$ by $\lambda$, the completed $L$-function $\Lambda(f_{49},s)$ of $f_{49}$ satisfies the functional equation 
\[
\Lambda(f_{49},s) = 7^s\Gamma(s)(2\pi)^{-s}L(f_{49},s) = -\lambda\Lambda(f_{49},2-s).
\]
So if $L(f_{49},1)\neq 0$ we must have $\lambda=-1$. The expansion at the cusp $1/7$ is given by
\begin{align*}
f_{49}|\gamma_7 = \frac{1}{7}\left((-2 \zeta_{7}^{5} - 4 \zeta_{7}^{4} - 6 \zeta_{7}^{3} - 8 \zeta_{7}^{2} - 3 \zeta_{7} - 5
)q + (6 \zeta_{7}^{5} - 2 \zeta_{7}^{4} + 4 \zeta_{7}^{3} + 3 \zeta_{7}^{2} + 2 \zeta_{7} + 1)q^2 + O(q^3)\right)
\end{align*}
\item $N = 8, ~k=16$. There are two rational newforms in $\mathcal{S}_{16}(8)$,
\[
f_{16,1} = q - 3444 q^3 + 313358q^5 +O(q^7)\text{ and }
f_{16,2} = q + 2700 q^3 - 251890q^5+O(q^7).
\]
Using products of Eisenstein series we find that both have $W_8$-eigenvalue $-1$. The expansions at at cusps of $\Gamma_0(8)$ other than $0$ and $\infty$ are
\begin{align}
f_1|\gamma_2 &=\frac{i}{256}\left(q_2 +3444q_2^3 + 313358 q_2^5
 +O(q_2^7)\right),~f_1|\gamma_4 = -f_1,\\
 f_2|\gamma_2 &=\frac{i}{256}\left(q_2 -2700 q_2^3-251890 q_2^5
 +O(q_2^7)\right),~f_2|\gamma_4 = -f_2.
\end{align}
We see that $f_1|\gamma_2B_2 = if_{1,\chi_4}$ and $f_2|\gamma_2B_2 = if_{2,\chi_4}$, where $\chi_4$ is the primitive character of conductor $4$.

\item $N= 36,~ k=8$. There is one newform $f_{36} = q - 270q^5 + O(q^6)\in
\mathcal{S}_8(36)$. The Atkin-Lehner operators are $W_{\{2\}}= \sabcd{1}{1}{-9}{-8}B_4$ and $W_{\{3\}} = \sabcd{1}{1}{8}{9}B_9$. 
\[
f|W_{36} = f_{36},~ f_{36}|W_{\{2\}} = -f_{36},~ f_{36}|W_{\{3\}} = -f_{36}.
\]
\item $N=3^5 = 243,~k=4$. We find the first few coefficients for $f_{243} = 
q -3q^2 +q^4+ 3q^5 -10q^7 + 21q^8 + O(q^{10})$ at the cusps $\frac13, \frac19$, $\frac{1}{27}$, and $\frac{1}{81}$.
\begin{align*}
f_{243}|\gamma_3 &= \frac{1}{729}\left((-\zeta_{162}^{29} + \zeta_{162}^{2})q_{27} -3 \zeta_{162}^{31} q_{27}^2 + (-\zeta_{162}^{35} + \zeta_{162}^{8})q_{27}^4 + 3 \zeta_{162}^{37} q_{27}^5
+ O(q_{27}^{6})\right),\\
f_{243}|\gamma_9 &= \frac19\left((\zeta_{54}^{14} - \zeta_{54}^{5})q_3
+3(- \zeta_{54}^{10} +  \zeta_{54})q_3^2
+\zeta_{54}^{11} q_3^4 
+3 \zeta_{54}^{7}q_3^5 + O(q_3^6)\right),\\
f_{243}|\gamma_{27} &=\zeta_{9}q - 3 \zeta_{9}^{2}q^{2} + \zeta_{9}^{4}q^{4} + 3 \zeta_{9}^{5}q^{5}+O(q^6),\\
f_{243}|\gamma_{81} &= -(\zeta_{3} + 1)q -3 \zeta_{3}q^2 -(\zeta_{3} + 1)q^4 +3 \zeta_{3} q^5+ O(q^6).
\end{align*}
\end{enumerate}

\textbf{Acknowledgements.}  We thank Andrew Booker for asking the off-the-cuff question that led us to study this problem, and for helpful remarks. We are grateful to Fredrik Str\"omberg for many remarks and corrections, Christian Wuthrich for discussions on Fourier coefficients of newforms at cusps, and Henri Cohen for providing us with formulas for the Fourier coefficients of Eisenstein series at cusps. 
We further thank Nikos Diamantis, Winfried Kohnen, and Martin Raum for their comments on the paper.
The authors were funded by the EPSRC grants EP/M016838/1 ``Arithmetic of hyperelliptic curves" and EP/N007360/1 ``Explicit methods for Jacobi forms over number fields" respectively.

\bibliographystyle{plain}
\bibliography{refs}

\end{document}